\documentclass{amsart}

\usepackage{amsmath}
\usepackage{amsfonts}
\usepackage{amssymb}
\usepackage{amscd}
\usepackage{amsthm}
\usepackage[mathscr]{eucal}
\theoremstyle{plain}
\theoremstyle{definition}
\newtheorem{theorem}{Theorem}[section]
\newtheorem{lemma}[theorem]{Lemma}
\newtheorem{proposition}[theorem]{Proposition}
\newtheorem{corollary}[theorem]{Corollary}
\newtheorem{definition}[theorem]{Definition}
\newtheorem{remark}[theorem]{Remark}

\newtheorem{example}[theorem]{Example}
\newtheorem{notation}[theorem]{Notation}
\newtheorem*{claim}{Claim}

\newcommand{\mc}[1]{{\mathcal #1}}

\newcommand{\seq}[2]{\ensuremath{({#1}_{#2})_{#2\in\mathbb{N}}}}
 \DeclareMathOperator{\supp}{supp}

\DeclareMathOperator{\conv}{conv}

\DeclareMathOperator{\spann}{span}
 \DeclareMathOperator{\ran}{ran}
\DeclareMathOperator{\Ker}{Ker} 

\DeclareMathOperator{\dist}{dist} 
\DeclareMathOperator{\sgn}{sgn} \DeclareMathOperator{\Ext}{Ext}
\DeclareMathOperator{\diag}{diag}
\newcommand{\R}{\mathbb{R}}
\newcommand{\N}{\mathbb{N}}
\newcommand{\Q}{\mathbb{Q}}

\newcommand{\e}{\varepsilon}
\newcommand{\co}{\ensuremath{c_{00}(\N)}}
\newcommand{\eqs}{{\mathfrak X}}

\theoremstyle{plain}

\begin{document}

\title{ Hereditarily Indecomposable Banach algebras of diagonal operators}

\author{Spiros A. Argyros}
\address[S.A. Argyros]{Department of Mathematics,
National Technical University of Athens}
\email{sargyros@math.ntua.gr}

\author{Irene  Deliyanni}
\address[I. Deliyanni]{18 Neapoleos St., Ag. Paraskevi, Athens}
 \email{ideliyanni@yahoo.gr}

\author{Andreas G. Tolias}
\address[A. Tolias]{Department of Mathematics,
 University of the Aegean}
\email{atolias@math.aegean.gr}

 \begin{abstract}
We provide a characterization of the Banach spaces $X$ with a
Schauder basis \seq{e}{n} which have the property that the dual
space $X^*$ is naturally isomorphic to the space
$\mc{L}_{\diag}(X)$ of diagonal operators with respect to
\seq{e}{n}. We also construct a Hereditarily Indecomposable Banach
space $\eqs_D$ with a Schauder basis \seq{e}{n} such that
$\eqs_D^*$ is isometric to  $\mc{L}_{\diag}(\eqs_D)$  with these
Banach algebras being Hereditarily Indecomposable. Finally, we
show that every $T\in \mc{L}_{\diag}(\eqs_D)$ is of the form
$T=\lambda I+K$, where $K$ is a compact operator.
 \end{abstract}


\footnotetext{ Research  partially supported by $\Pi EBE$ 2007
  NTUA Research Program.}


\maketitle

\tableofcontents

\section{Introduction}  The starting point of this paper is a result connecting the dual space $X^*$ of a
space $X$ with a Schauder basis  \seq{e}{n}, with the space
$\mc{L}_{\diag}(X)$ of the diagonal operators  with respect to
this basis. We recall that $\mc{L}_{\diag}(X)$ is the commutative
subalgebra of $\mc{L}(X)$ containing all bounded linear operators
$T$ satisfying $Te_n=\lambda_n e_n$, $n\in\N$, for a sequence
$(\lambda_n)_{n\in\N}$ of scalars. As is well known, if the basis
\seq{e}{n} is unconditional, the algebra $\mc{L}_{\diag}(X)$ is
homeomorphic to the algebra $\ell_{\infty}(\N)$.
 Our result asserts that, under
some natural assumptions, the spaces  $X^*$ and
$\mc{L}_{\diag}(X)$ are isomorphic. There are classical spaces,
such as the space $c(\N)$ of all convergent sequences with the
summing basis, that satisfy these conditions and thus the
structure of the space of the diagonal operators acting on them is
completely described. In particular, for the space $X=c(\N)$ with
the summing basis, we obtain that $\mc{L}_{\diag}(X)$ is isometric
to $\ell_1(\N)$. Our first main result is the following.

 \begin{theorem}\label{thA}
  Let $X$ be a Banach space with a Schauder basis \seq{e}{n}. The
  following are equivalent.
\begin{enumerate}
  \item[(1)] The map $e_n^* \mapsto e_n^* \otimes e_n$ can be extended
  to an isomorphism between $X^*$ and $\mc{L}_{\diag}(X)$.
  \item[(2)] \begin{enumerate}
  \item[(a)] The basis \seq{e}{n} dominates the summing
  basis.
  \item[(b)]   The norm in $X^*$ is submultiplicative (i.e. there exists $C>0$
  such that $\|\sum\limits_{i=1}^na_i\beta_ie_i^*\|\le
  C\cdot \|\sum\limits_{i=1}^na_ie_i^*\|\cdot
  \|\sum\limits_{i=1}^n\beta_ie_i^*\|$.)
\end{enumerate}
\end{enumerate}
 \end{theorem}
The above theorem essentially concerns conditional bases of Banach
spaces. Indeed, assuming that \seq{e}{n} is an unconditional
basis, condition (2)(a) yields that \seq{e}{n} is equivalent to
the standard basis of $\ell_1(\N)$.

As a consequence of Theorem 1.1 we obtain the following.

\begin{theorem}\label{thB}
For every Banach space $Z$ with an unconditional subsymmetric
basis \seq{z}{n} there exists a Banach space $X$ with a Schauder
basis \seq{e}{n} such that $Z^*$ is isomorphic to a complemented
subspace of $\mc{L}_{\diag}(X)$.
\end{theorem}

Theorem \ref{thA} also yields the existence of a variety of Banach
spaces  $X$ with a Schauder basis \seq{e}{n} sharing the property
that $X^*$ is isomorphic to $\mc{L}_{\diag}(X)$. In particular,
using a slight modification of the James tree space (\cite{J}), we
obtain the following.

\begin{theorem}\label{thC}
There exists a Banach space $X$  with a Schauder basis \seq{e}{n}
such that $X^*$ is isomorphic to $\mc{L}_{\diag}(X)$, $X^*$ is
nonseparable and does not contain $\ell_1(\N)$ or $c_0(\N)$.
\end{theorem}

Let us also mention that A. Sersouri has shown in \cite{Se} that
if the basis \seq{e}{n} of the Banach space $X$ is either
boundedly complete or shrinking, then $\mc{L}_{\diag}(X)$
coincides with the second dual of the space  $\mc{K}_{\diag}(X)$
of compact diagonal operators.

In the second part of the paper we construct a Hereditarily
Indecomposable (HI) Banach space $\eqs_D$ with a Schauder basis
\seq{e}{n} such that $\mc{L}_{\diag}(\eqs_D)$ is also HI. More
precisely, the following is shown.

 \begin{theorem}\label{thD}
  There exists a quasi reflexive HI Banach space $\eqs_D$ with a Schauder basis
  \seq{e}{n} satisfying the following.
 \begin{enumerate}
  \item[(i)] The  space $\eqs_D^*$ is Hereditarily Indecomposable.
  \item[(ii)] The spaces $\eqs_D^*$ and $\mc{L}_{\diag}(\eqs_D)$
  are isometric.
  \item[(iii)] Every $T\in \mc{L}_{\diag}(\eqs_D)$ is
  of the form
  $T=\lambda I+K$, where $K$ is a compact operator.
 \end{enumerate}
 \end{theorem}

As  pointed out in \cite{AH},  for every  Banach space $X$, the
space $X^*$ is isomorphic to  a complemented subspace of
$\mc{L}(X)$ consisting of rank one operators. Therefore neither
$\mc{L}(X)$ nor $\mc{K}(X)$ can be indecomposable spaces. Also, as
shown in \cite{ADT}, there exist HI spaces having strictly
singular non-compact diagonal operators, therefore one could not
expect property (iii) to hold in general within the class of HI
spaces with a Schauder basis. Recently, R. Haydon and the first
named author (\cite{AH}) have presented a $\mc{L}^{\infty}$ HI
space $\eqs_K$ such that every $T\in \mc{L}(\eqs_K)$ is of the
form $T=\lambda I +K$ with $K$ a compact operator.
  However, the scalar plus compact
 problem remains open for reflexive Banach spaces. The weaker question whether
 there exists a reflexive space $X$ with a
 Schauder basis \seq{e}{n} such that every diagonal operator $T$
 is of the form $\lambda I+K$ where $K$ is compact, is also open
 and its solution could be a first step towards the solution of the general one.
 The systematic study of the existence of strictly singular
 non-compact operators on the known HI spaces (see \cite{AnSc},
 \cite{ADT}, \cite{Be}, \cite{Ga}, \cite{G1}), indicates that an
 example of a reflexive space answering the scalar plus compact problem, in
  the general or the weaker form, requires new approaches of HI
  constructions.

The paper is organized as follows. Section 2 contains the more
precise statement of Theorem \ref{thA} (Theorem \ref{th1}). Its
proof uses rather standard arguments. We also present the proofs
of Theorems \ref{thB} and \ref{thC}.

The rest of the paper is devoted to the space $\eqs_D$. In section
3, we define its norming set $D$ which is a subset of
$c_{00}(\N)$. The space $\eqs_D$ is the completion of
$(c_{00}(\N),\|\;\|_D)$, where $\|\;\|_D$ is the norm induced on
$c_{00}(\N)$ by $D$. For the definition of $D$ we use as a ground
set the set $G=\{\pm \chi_I:\; I \mbox{ finite interval of }\N\}$
and we apply saturation with respect to the operations
$(\mc{A}_{n_j},\frac{1}{m_j})_j$. As usual, for  even indices $j$
we apply full saturation, while the operations corresponding to
odd indices $j$ are used in order to define the special
functionals as in all previous HI constructions, initialized by
the W.T. Gowers and B. Maurey example \cite{GM1}. Once more, the
present construction can be viewed as an HI extension of a ground
set $G$, according to the approach of \cite{AT1} and \cite{ArTo}.

The novelty of the present construction arises from the need to
impose a Banach algebra structure on the dual of the space. For
this purpose, in each inductive step of the definition of the
norming set $D=\bigcup\limits_{n=0}^{\infty}D_n$, we close the set
$D_n$ under the pointwise products of its elements. This is
necessary in order to apply Theorem \ref{thA} and get the
isomorphism between $\eqs_D^*$ and $\mc{L}_{\diag}(\eqs_D)$. As
shown in \cite{AT1}, for a ground set $G\subset c_{00}(\N)$ such
that $X_G$ does not contain any isomorphic copy of $\ell_1(\N)$,
there exists a $D_G\subset c_{00}(\N)$ containing $G$ such that
$X_{D_G}$ is HI. It is worth pointing out that there are ground
sets $G$ as above, such that for any $D_G$ containing $G$ with
$D_G$ closed under pointwise products of its elements, the
corresponding space $X_{D_G}$ is decomposable and hence is not HI.
For example, let $L\subset \N$ such that both $L,\N\setminus L$
are infinite and consider the ground set $G=\{\pm \chi_{I},\; \pm
\chi_{L\cap I},\;I \mbox{ finite interval of }\N \}$. Then for
every extension $D_G$ of $G$ with $D_G$ being closed under
pointwise products, the corresponding space $X_{D_G}$ is
decomposable. Indeed, it is easy to see that the subspace
$X_L=\overline{\spann}\{e_n:\;n\in L\}$ is complemented in
$X_{D_G}$.

Section 4 is devoted  to the basic inequality and some of its
consequences. The basic inequality is the main tool for providing
upper estimates for the action of functionals of $D$ on averages
of Rapidly Increasing Sequences. Its proof in the present paper is
similar to the proof of the corresponding result in \cite{ADT}. In
section 5, we proceed to the evaluation of the norm of averages
resulting from dependent sequences with alternating signs. Our
approach for this, is more complicated  than in previous
constructions, where this result is a direct consequence of the
basic inequality and the tree structure of the special sequences.
 This is due to the
fact that closing the set $D$ under pointwise products of its
elements, we enlarge the unconditional structure of the spaces
$\eqs_D$, $\eqs_D^*$. Thus showing the HI property of the space
$\eqs_D$ becomes more involved and delicate.

The HI property of $\eqs_D^*$ is proved in section 6. The
isomorphism between $\eqs_D^*$  and
$\mc{L}_{\diag}(\eqs_D,\seq{e}{n})$ is an immediate consequence of
Theorem \ref{thA} and the definition of the norming set $D$.

We close this section  by pointing out that it is not clear if,
for every Schauder basis \seq{x}{n} of $\eqs_D$, the corresponding
space $\mc{L}_{\diag}(\eqs_D,\seq{x}{n})$  is isomorphic to
$\eqs_D^*$ or if it is Hereditarily Indecomposable.

 \section{On the isomorphism between  $\mc{L}_{\diag}(X)$ and  $X^*$}

In this section, we give the precise statement and the proof of
the characterization of the Banach spaces $X$ with a Schauder
basis $(e_n)_{n\in {\mathbb N}}$ which have the property that the
dual space $X^*$ is naturally isomorphic the space
$\mc{L}_{\diag}(X)$ of the diagonal operators with respect to
$\seq{e}{n}$. We also state and prove Theorems \ref{thB} and
\ref{thC}.

 We start with some preliminary notation.
 We denote by $c_{00}(\N)$ the space of all eventually zero
 sequences of reals, by $e_1,e_2,\ldots$ its  standard Hamel
 basis, while for $x=\sum a_i e_i\in
 c_{00}(\N)$ the support  $\supp x$ of $x$ is the set
 $\supp x=\{i\in\N:\;a_i\neq 0\}$.
  For $E,F$ nonempty finite subsets of $\N$, we
 write $E<F$ if $\max E<\min F$, while for nonzero $x,y\in
 c_{00}(\N)$ we write $x<y$ if $\supp x<\supp y$.
 For $x,y\in c_{00}(\N)$,
 $x=\sum a_ie_i$, $y=\sum \beta_ie_i$ the pointwise product of $x,y$
 is the vector $x\cdot y=\sum\limits a_i \beta_i e_i$.
 For $f\in c_{00}(\N)$ and $E\subset\N$ we denote by $Ef$ the pointwise
 product $\chi_E\cdot f$.
 For a finite set $F$, we denote its cardinality by $\# F$. For
 $K,L\subset c_{00}(\N)$ we denote $K+ L=\{f+g:\;f\in K,\;g\in
 L\}$ and $K\cdot L=\{f\cdot g:\;f\in K,\;g\in
 L\}$.

\begin{notation}\label{not1}
Let $(X,\|\cdot\|)$ be a Banach space with a Schauder basis
\seq{e}{n}. We denote by $\mc{L}_{\diag}(X,\seq{e}{n})$ the algebra
of all bounded linear diagonal operators of $X$ with respect to the basis
\seq{e}{n}, i.e.
\begin{equation*}
  \mc{L}_{\diag}(X,\seq{e}{n})=\{T\in\mc{L}(X):\; \exists (\lambda_n)_{n\in\N}\in\R^\N\mbox{
  such that }
  Te_n=\lambda_ne_n,\;\forall n\}.
\end{equation*}
When the basis  is a priori fixed we use the notation
$\mc{L}_{\diag}(X)$.

 For every $n$, we denote by $\mathbf{\overline{e}}_n$  the rank one
 operator $\mathbf{\overline{e}}_n=e_n^*\otimes e_n:X\to X$, i.e. the
 diagonal operator
 defined by the rule
 $\mathbf{\overline{e}}_n(\sum\limits_{i=1}^{\infty}\mu_ie_i)=\mu_ne_n$.
 \end{notation}

 \begin{remark}
 In the case the basis \seq{e}{n} of the space $X$ is unconditional, the
 algebra $\mc{L}_{\diag}(X)$  is isomorphic to $\ell_{\infty}(\N)$.
 \end{remark}

 \begin{definition}
 Let $X$ be a Banach space with a Schauder basis \seq{e}{n}. We
 say that
 \seq{e}{n} dominates the summing basis with constant $C_1$, if
  for every finite sequence of scalars
 $(\mu_i)_{i=1}^n$ it holds  that $C_1\cdot
 |\sum\limits_{i=1}^n\mu_i|\le \|\sum\limits_{i=1}^n\mu_ie_i\|$.
 \end{definition}

 \begin{theorem}\label{th1}
 Let $(X,\|\cdot\|)$ be a Banach space with a normalized monotone
 Schauder basis \seq{e}{n}. Let also $C_1,C_2>0$.
 The statements (1), (2), (3) are equivalent:\\
 (1)\; The operator $\Phi:X^*\to \mc{L}_{\diag}(X)$ defined by the rule
    \[
     w^*-\sum\limits_{n=1}^{\infty}\lambda_n e_n^* \;\;
     \longmapsto\;\;
       SOT-\sum\limits_{n=1}^{\infty}\lambda_n
  \mathbf{\overline{e}}_n
     \]
     is well defined,
  onto and $C_1\cdot\|\sum\limits_{n=1}^{\infty}\lambda_n e_n^*\|
  \le \|\sum\limits_{n=1}^{\infty}\lambda_n
  \mathbf{\overline{e}}_n\|\le
  C_2\cdot\|\sum\limits_{n=1}^{\infty}\lambda_n e_n^*\|$ for every
   $x^*=w^*-\sum\limits_{n=1}^{\infty}\lambda_ne_n^*\in X^*$. \\
 (2)\; (a)\; The Schauder basis \seq{e}{n}  dominates the summing
 basis with constant $C_1$.

 (b)\; The norm of $X^*$ is submultiplicative with constant $C_2$, i.e.
 \[\|\sum\limits_{i=1}^na_i\beta_ie_i^*\|\le
 C_2\cdot\|\sum\limits_{i=1}^na_ie_i^*\|\cdot
 \|\sum\limits_{i=1}^n\beta_ie_i^*\|\] for every $n\in \N$ and
 every choice of scalars
 $a_1,\beta_1,a_2,\beta_2,\ldots,a_n,\beta_n\in \R$.\\
 (3)\; There exists a 1-norming set $K$ of $X$,
  contained in the linear span of $(e_n^*)_{n\in\N}$,
 such that

 (a)\; $\pm C_1 \cdot\sum\limits_{i=1}^{n}e_i^*\in K$ for all $n$.

 (b)  $K\cdot K\subset C_2\cdot B_{X^*}$.
 \end{theorem}
 \begin{proof}[\bf Proof.]
 We first show that $(2)\Longrightarrow(3)$. Suppose that (2) holds and set \[K=\big\{\sum\limits_{i=1}^na_ie_i^*:\;
 \|\sum\limits_{i=1}^na_ie_i^*\|\le 1, \;n\in\N\big\}.\] It is clear
 that (3)(b) is satisfied, as a consequence of (2)(b).
 Let us verify (3)(a).  For every $x=\sum\limits_{i=1}^{\infty}\mu_ie_i\in B_X$,
 using condition (2)(a) and the monotonicity of the Schauder basis $(e_i)_{i\in\N}$,
 we get that
  \[ |(\sum\limits_{i=1}^ne_i^*)(x)|=|\sum\limits_{i=1}^n\mu_i|\le
  \frac{1}{C_1}\cdot \|\sum\limits_{i=1}^n\mu_ie_i\|\le
  \frac{1}{C_1}\cdot \|\sum\limits_{i=1}^\infty\mu_ie_i\|\le\frac{1}{C_1}.\]
  This implies that $\|\sum\limits_{i=1}^ne_i^*\|\le\frac{1}{C_1}$, hence
  $\pm C_1\cdot\sum\limits_{i=1}^ne_i^*\in K$.

 Let us show the inverse implication $(3)\Longrightarrow(2)$.
 Suppose that (3) holds and let $n\in\N$, $\mu_1,\ldots,\mu_n\in\R$. Since
 $\pm C_1\sum\limits_{i=1}^ne_i^*\in K$
  the action of these functionals on the vector $\sum\limits_{i=1}^n\mu_ie_i$ implies that
  $\|\sum\limits_{i=1}^n\mu_ie_i\|\ge C_1\cdot |\sum\limits_{i=1}^n\mu_i|$, thus (2)(a) is
  satisfied. From condition (3)(b) we get that
  $\conv(K)\cdot \conv(K)\subset C_2\cdot B_{X^*}$ and hence
 $\overline{\conv(K)}^{w^*}\cdot \overline{\conv(K)}^{w^*}
 \subset C_2 \cdot B_{X^*} $.
  Since $K$ is a 1-norming set of the space $X$, this
  means that $B_{X^*}\cdot B_{X^*}\subset C_2\cdot B_{X^*}$, which yields (2)(b).

 Next we show the implication $(1)\Longrightarrow(2)$. Suppose
 that (1) is true. We shall first show
 (2)(a). We observe that for every $n\in\N$ and $x=\sum\limits_{i=1}^{\infty}\mu_ie_i\in X$ we
 have that
 $\|(\sum\limits_{i=1}^n \mathbf{\overline{e}}_i)(x)\|=\|\sum\limits_{i=1}^n \mu_ie_i\|\le \|x\|$,
 as a consequence of the monotonicity of the
 basis. Thus $\|\sum\limits_{i=1}^n \mathbf{\overline{e}}_i\|\le 1$,
   which
 implies that $C_1\|\sum\limits_{i=1}^n e_i^*\|\le 1$.
 Therefore for any $\mu_1,\ldots,\mu_n\in \R$ we get that
 \[
 C_1|\sum\limits_{i=1}^n\mu_i|=C_1|(\sum\limits_{i=1}^ne_i^*)(\sum\limits_{i=1}^n\mu_ie_i)|\le
 C_1\|\sum\limits_{i=1}^ne_i^*\|\cdot
 \|\sum\limits_{i=1}^n\mu_ie_i\|\le
 \|\sum\limits_{i=1}^n\mu_ie_i\|\]
 and we have shown (2)(a).

 Let now $n\in\N$ and  $a_1,\beta_1,a_2,\beta_2,\ldots,a_n,\beta_n\in \R$.
 We choose $x=\sum\limits_{i=1}^{\infty}\mu_ie_i\in B_X$, such that
 $\|\sum\limits_{i=1}^na_i\beta_ie_i^*\|
 =(\sum\limits_{i=1}^na_i\beta_ie_i^*)(x)
 =\sum\limits_{i=1}^na_i\beta_i\mu_i$.
 Then
 \begin{eqnarray*}
 \|\sum\limits_{i=1}^na_i\beta_ie_i^*\| & = &
       (\sum\limits_{i=1}^na_ie_i^*)(\sum\limits_{i=1}^n\beta_i\mu_ie_i)  \le
     \|\sum\limits_{i=1}^na_ie_i^*\| \cdot \|\sum\limits_{i=1}^n\beta_i\mu_ie_i\|  \\
  & = & \|\sum\limits_{i=1}^na_ie_i^*\| \cdot
         \| (\sum\limits_{i=1}^n\beta_i
         \mathbf{\overline{e}}_i)(\sum\limits_{i=1}^n\mu_ie_i)\|\\
     &     \le & \|\sum\limits_{i=1}^na_ie_i^*\| \cdot
   \| \sum\limits_{i=1}^n\beta_i \mathbf{\overline{e}}_i\|  \cdot     \|\sum\limits_{i=1}^n\mu_ie_i\|\\
  & \le  & \|\sum\limits_{i=1}^na_ie_i^*\| \cdot
    C_2 \| \sum\limits_{i=1}^n\beta_i e_i^*\| \cdot \|x\| \\
 & \le &
  C_2\cdot \|\sum\limits_{i=1}^na_ie_i^*\|\cdot  \| \sum\limits_{i=1}^n\beta_i e_i^*\|
  \end{eqnarray*}
 which completes the proof of (2)(b).

Finally we prove that   $(2)\Longrightarrow(1)$. Suppose that (2)
holds. We start with the following claim.
\begin{claim}
The series $\sum\limits_{n=1}^{\infty}\lambda_ne_n^*$ is $w^*$
convergent in $X^*$ if and only if the series
$\sum\limits_{n=1}^{\infty}\lambda_n\mathbf{\overline{e}}_n$
converges in the strong operator topology in $\mc{L}_{\diag}(X)$.
\end{claim}
\begin{proof}[\bf Proof of the claim.]
Suppose first that the series
$\sum\limits_{n=1}^{\infty}\lambda_ne_n^*$ is $w^*$ convergent. We
choose $M>0$ such that $\|\sum\limits_{i=n}^m\lambda_ie_i^*\|\le
M$ for all $n\le m$. We consider an arbitrary $x\in X$,
$x=\sum\limits_{i=1}^{\infty}\mu_ie_i$, and we shall show that the
sequence
$\big((\sum\limits_{i=1}^{n}\lambda_i\mathbf{\overline{e}}_i)(x)\big)_{n\in\N}$
\big(i.e. the sequence
$(\sum\limits_{i=1}^n\lambda_i\mu_ie_i)_{n\in\N}$\big) is a Cauchy
sequence in $X$. Let $\e>0$. We choose $n_0\in\N$ such that
$\|\sum\limits_{i=n_0}^{\infty}\mu_ie_i\|<\frac{\e}{MC_2}$. Let
now any $m\ge n\ge n_0$. We select $z^*\in B_{X^*}$, with
$z^*=\sum\limits_{i=1}^{\infty}\nu_ie_i^*$ as a $w^*$ series, such
that
 $\|\sum\limits_{i=n}^m\lambda_i\mu_ie_i\|=z^*(\sum\limits_{i=n}^m\lambda_i\mu_ie_i)
=\sum\limits_{i=n}^m\lambda_i\mu_i\nu_i$. We set
$f=\sum\limits_{i=n}^m\lambda_i\nu_ie_i^*$. From our assumption
(2)(b) we get that $\|f\|\le
C_2\cdot\|\sum\limits_{i=n}^m\lambda_ie_i^*\|\cdot\|\sum\limits_{i=n}^m\nu_ie_i^*\|
\le C_2M \|z^*\|\le C_2 M$. Thus
 $\|(\sum\limits_{i=n}^m\lambda_i\mathbf{\overline{e}}_i)(x)\|
 =\|\sum\limits_{i=n}^m\lambda_i\mu_ie_i\|
 =\sum\limits_{i=n}^m\lambda_i\mu_i\nu_i
 =f(\sum\limits_{i=n}^m \mu_ie_i)
 \le\|f\|\cdot\|\sum\limits_{i=n}^m \mu_i e_i\|
 <C_2M\frac{\e}{MC_2}=\e$.

 Conversely, suppose that the series
 $\sum\limits_{n=1}^{\infty}\lambda_n\mathbf{\overline{e}}_n$ converges
 in the strong operator topology; we shall prove that the series
  $\sum\limits_{n=1}^{\infty}\lambda_n e_n^*$ is $w^*$ convergent.
  Let $x\in X$, $x=\sum\limits_{i=1}^{\infty}\mu_ie_i$. We shall show that
  the sequence $\big( (\sum\limits_{i=1}^{n}\lambda_i  e_i^*)(x)\big)_{n\in\N}$
  \big(i.e. the sequence $(\sum\limits_{i=1}^n\lambda_i\mu_i)_{n\in \N}$ \big) is a Cauchy sequence
  in $\R$.  Let $\e>0$. From our assumption that the series
  $\sum\limits_{n=1}^{\infty}\lambda_n\mathbf{\overline{e}}_n$ is
 SOT-convergent it follows that the
  sequence $\big((\sum\limits_{i=1}^{n}\lambda_i\mathbf{\overline{e}}_i)(x)\big)_{n\in\N}$
 converges in norm. Thus
 we may choose $n_0$ such that
 $\|(\sum\limits_{i=n}^m\lambda_i\mathbf{\overline{e}}_i)(x)\|<C_1\e$ for
 every $m\ge n\ge n_0$.
 From assumption (2)(a) we get that
 $|\sum\limits_{i=n}^m\lambda_i\mu_i|\le
 \frac{1}{C_1}\|\sum\limits_{i=n}^m\lambda_i\mu_ie_i\|
 =\frac{1}{C_1}\|(\sum\limits_{i=n}^m\lambda_i\mathbf{\overline{e}}_i)(x)\|<\frac{1}{C_1}C_1\e=\e$.
 This completes the proof of the claim.
 \end{proof}
   From the first part of the claim it follows that the operator
   $\Phi:X^*\to\mc{L}_{\diag}(X)$ is well defined.
   Taking into account that for $T\in \mc{L}_{\diag}(X)$,
   if $T(e_n)=\lambda_ne_n$, $n=1,2,\ldots$ then
   $T=SOT-\sum\limits_{n=1}^{\infty}\lambda_n\mathbf{\overline{e}}_n$,
    the second part of the claim entails that
   the operator $\Phi$ is onto $\mc{L}_{\diag}(X)$.
   We shall show that
   $C_1\cdot\|\sum\limits_{i=1}^{\infty}\lambda_ie_i^*\|
   \le \|\sum\limits_{i=1}^{\infty}\lambda_i\mathbf{\overline{e}}_i\|
   \le C_2\cdot \|\sum\limits_{i=1}^{\infty}\lambda_ie_i^*\|$ for
   every $x^*=w^*-\sum\limits_{i=1}^{\infty}\lambda_ie_i^*\in X^*$.
 From now on we fix a functional $x^*=w^*-\sum\limits_{i=1}^{\infty}\lambda_ie_i^*\in X^*$.

 From (2)(a) we get that
 \begin{eqnarray*}
  \|\sum\limits_{i=1}^{\infty}\lambda_i\mathbf{\overline{e}}_i\| & = &
   \sup\{\|(\sum\limits_{i=1}^{\infty}\lambda_i\mathbf{\overline{e}}_i)(\sum\limits_{i=1}^{\infty}\mu_ie_i)\|:\;
    \sum\limits_{i=1}^{\infty}\mu_ie_i\in B_X\}\\
   & = &   \sup\{\|\sum\limits_{i=1}^{n}\lambda_i\mu_ie_i\|:\;
    \sum\limits_{i=1}^{\infty}\mu_ie_i\in B_X,\;n\in\N\}\\
  & \ge & \sup\{C_1\cdot |\sum\limits_{i=1}^n\lambda_i\mu_i|:\;
   \sum\limits_{i=1}^{\infty}\mu_ie_i\in B_X,\;n\in\N\}\\
  & = &   C_1 \cdot \sup\{ |(\sum\limits_{i=1}^{\infty}\lambda_ie_i^*)(\sum\limits_{i=1}^n\mu_ie_i)|:\;
     \sum\limits_{i=1}^{\infty}\mu_ie_i\in B_X,\;n\in\N\}\\
  & = &   C_1 \cdot \|\sum\limits_{i=1}^{\infty}\lambda_ie_i^*\|.
 \end{eqnarray*}

 We finally show that
  $\|\sum\limits_{i=1}^{\infty}\lambda_i\mathbf{\overline{e}}_i\|
   \le C_2\cdot \|\sum\limits_{i=1}^{\infty}\lambda_ie_i^*\|$.
     Let $\e>0$. We
 select a finitely supported vector $x\in B_X$,
 $x=\sum\limits_{i=1}^n\mu_ie_i$, such that
 $\|\sum\limits_{i=1}^{\infty}\lambda_i\mathbf{\overline{e}}_i\|\le
 (1+\e)\|(\sum\limits_{i=1}^{\infty}\lambda_i\mathbf{\overline{e}}_i)(x)\|
 =(1+\e)\|\sum\limits_{i=1}^n\lambda_i\mu_ie_i\|$.
 We choose $z^*=w^*-\sum\limits_{i=1}^{\infty}\nu_ie_i^*\in B_{X^*}$
  such that
  $\|\sum\limits_{i=1}^n\lambda_i\mu_ie_i\| =z^*(\sum\limits_{i=1}^n\lambda_i\mu_ie_i)
  =\sum\limits_{i=1}^n\lambda_i\mu_i\nu_i$.
  We set $f=\sum\limits_{i=1}^n\lambda_i\nu_ie_i^*$.
 From our assumption (2)(b) we get that
 $\|f\|\le C_2
 \cdot\|\sum\limits_{i=1}^n\lambda_ie_i^*\|\cdot\|\sum\limits_{i=1}^n\nu_ie_i^*\|
 \le C_2 \cdot\|\sum\limits_{i=1}^{\infty}\lambda_ie_i^*\|$.
 Therefore
 \begin{eqnarray*}
  \|\sum\limits_{i=1}^{\infty}\lambda_i\mathbf{\overline{e}}_i\|
  & \le & (1+\e)\cdot \|\sum\limits_{i=1}^n \lambda_i\mu_i e_i\|
        =  (1+\e)\cdot  \sum\limits_{i=1}^n\lambda_i\mu_i\nu_i \\
        & = & (1+\e)\cdot f(x) \le  (1+\e)\cdot \|f\|\cdot\|x\|
  \le  (1+\e)\cdot C_2\cdot \|\sum\limits_{i=1}^{\infty}\lambda_ie_i^*\|.
 \end{eqnarray*}
 Since this happens for every $\e>0$ we conclude that
 $\|\sum\limits_{i=1}^{\infty}\lambda_i\mathbf{\overline{e}}_i\| \le
 C_2 \cdot \|\sum\limits_{i=1}^{\infty}\lambda_ie_i^*\|$ and this finishes the proof of the theorem.
\end{proof}

 \begin{remark}\label{sa12}
 Usually a $K\subset c_{00}(\N)$ is considered and the space $X$ is defined as the completion
 of the normed space $(c_{00}(\N),\|\;\|_K)$. If for such a $K$ it holds that
  $K\cdot K\subset K+\cdots +K$ ($m$ summands) then condition (3)(b) is satisfied for $C_2=m$.
 \end{remark}

 \begin{remark}\label{Nrem0007}
  One can easily prove that under the conditions of Theorem \ref{th1},  for every choice of
  scalars $(a_i)_{i\in\N}$, the series $\sum\limits_{i=1}^{\infty}a_ie_i^*$ converges in norm
  in $X^*$
  if and only if the series $\sum\limits_{i=1}^{\infty}a_i\mathbf{\overline{e}}_i$
  converges in norm in $\mc{L}_{\diag}(X)$. Since the latter means that
  the operator $\sum\limits_{i=1}^{\infty}a_i\mathbf{\overline{e}}_i$  is
  compact, we get that under the conditions of Theorem \ref{th1},
  the   space $\mc{K}_{\diag}(X)$ of compact diagonal
  operators of $X$ is naturally identified with the subspace of
  $X^*$
   norm generated by the
  biorthogonal functionals $(e_n^*)_{n\in\N}$.
 \end{remark}

\begin{example}
  The summing
 basis \seq{s}{n} of the
  Banach space $c(\N)$ of all convergent sequences is monotone while
  the  set $K=\big\{\pm\sum\limits_{i=1}^ns_i^*:\; n\in\N\big\}$
  is a norming set of the space $c(\N)$
   satisfying conditions (3)(a), (3)(b) of Theorem
 \ref{th1} with constants $C_1=1$ and $C_2=1$.
   Therefore, Theorem \ref{th1}
  implies that the space of all diagonal operators of the space $c(\N)$
  with respect to the summing basis,
  is isometric to $c(\N)^*$, which is isometric to $\ell_1(\N)$.
  \end{example}

 \begin{remark}
 It follows readily that if the space $X$ has an unconditional
 basis \seq{e}{n}, then, denoting by $(e_n^*)_{n\in\N}$ the
 corresponding biorthogonal functionals,  the following holds.
 For every $f=w^*-\sum\limits_{n=1}^{\infty}a_ne_n^*$,
 $g=w^*-\sum\limits_{n=1}^{\infty}\beta_ne_n^*$ in $X^*$ we have
 that
 \[ \|w^*-\sum\limits_{n=1}^{\infty}a_n\beta_ne_n^*\|\le C\cdot
    \|w^*-\sum\limits_{n=1}^{\infty}a_ne_n^*\|\cdot
    \|w^*-\sum\limits_{n=1}^{\infty}\beta_ne_n^*\|
    \]
    where $C$ is a constant which depends on the unconditional basis constant of \seq{e}{n}.
Therefore the dual of a space with an unconditional basis is a
Banach algebra.

On the other hand, it is clear that if the space $X$ has an
unconditional basis \seq{e}{n} and it satisfies condition (2)(a)
of Theorem \ref{th1}, then $X$ is isomorphic to $\ell_1(\N)$.
 Thus, although for every space $X$ with an unconditional basis
 \seq{e}{n} it holds that $\mc{L}_{\diag}(X,\seq{e}{n})$ is
 isomorphic to $\ell_{\infty}(\N)$, the only space for which this
 fact follows as a consequence of Theorem \ref{th1}, is $\ell_1(\N)$.
 \end{remark}

 \begin{theorem}
 Let $Z$ be a Banach space with an unconditional subsymmetric
 Schauder basis. Then there exists a Banach space $X$ with a
 Schauder basis \seq{e}{n} such that
  $\mc{L}_{\diag}(X,\seq{e}{n})$ has a complemented subspace
  isomorphic to $Z^*$.
 \end{theorem}
 \begin{proof}[\bf Proof.]
 Passing to an equivalent norm, we may assume that the space
 $Z$ has a normalized, bimonotone basis \seq{z}{n} which
 is 1-unconditional and subsymmetric. This means that, given
 $z=\sum\limits_{i=1}^d\lambda_iz_i$,
 for any increasing sequence $(k_i)_{i=1}^d$ we have that
 $\|\sum\limits_{i=1}^d\lambda_iz_{k_i}\|=\|z\|$ and that for every
 choice of scalars $(\mu_i)_{i=1}^d$ with $|\mu_i|\le |\lambda_i|$
 we have that $\|\sum\limits_{i=1}^d \mu_i z_i\|\le \|z\|$.
 The same properties remain valid for the sequence of biorthogonal
 functionals $(z_n^*)_{n\in\N}$.

 The space $X$ is defined to be the Jamesification $J_Z$ of the space
 $Z$ (see also \cite{BHO}).
 Namely, setting
 \begin{eqnarray*}
 K & = & \big\{ \sum\limits_{i=1}^d a_i\chi_{I_i}:\;
 I_1<I_2<\cdots<I_d\mbox{ are finite intervals,}\\
  & & \;\qquad\qquad\qquad\qquad   \|\sum\limits_{i=1}^da_iz_i^*\|_{Z^*}\le 1,\; d\in\N\big\}
 \end{eqnarray*}
 the space $X=J_Z$ is the completion of $(c_{00}(\N),\|\;\|_K)$.
 The Hamel basis \seq{e}{n} of $c_{00}(\N)$ becomes a normalized
 bimonotome Schauder basis for $X$, while the norming set $K$ of
 $X$ obviously satisfies condition (3)(a) of Theorem \ref{th1}
 with constant $C_1=1$.

 We next show that $K\cdot K\subset K+K$. Fix $f=\sum\limits_{i=1}^d a_i\chi_{I_i}$,
 $g=\sum\limits_{j=1}^{d'} \beta_j\chi_{E_j}$  in $K$.
 Without loss of generality we may assume that each $I_i$
 intersects some $E_j$ and and each $E_j$ intersects some $I_i$.
 For each $j=1,\ldots,d'$, let $i_j$ be the minimum $i$ for which
 $I_i\cap E_j\neq\emptyset$.
 Using  that $\|\sum\limits_{j=1}^{d'}\beta_jz_j^*\|\le
 1$, $|a_i|\le 1$ for each $i$,  the 1-unconditionality and the
 subsymmetricity of the basis \seq{z}{n}, we get that the
 functional
 $h_1=\sum\limits_{j=1}^da_{i_j}\beta_j \chi_{I_{i_j}\cap E_j}$ belongs to the set $K$.
 Observe that for each $i$, there exists at most
 one $j$ such that $I_i\cap E_j\neq\emptyset$ and $i\neq i_j$. Let $A$ be the set of
 all $i$ for which such a $j$ exists and  denote this $j$ by
 $j_i$. Using that $\|\sum\limits_{i=1}^{d}a_iz_i^*\|\le 1$,
 $|b_j|\le 1$ for each $j$,  the 1-unconditionality and the
 subsymmetricity of the basis \seq{z}{n}, we derive that
 the functional
 $h_2=\sum\limits_{i\in A}a_i b_{j_i}\chi_{I_i\cap E_{j_i}}$ also
 belongs to the set $K$.
 Hence the functional $f\cdot g=h_1+h_2$ belongs to $K+K$, therefore
 (see Remark \ref{sa12}) condition (3)(b) of Theorem \ref{th1} is satisfied with constant
 $C_2=2$.

   Theorem \ref{th1} entails that the space of diagonal operators
  $\mc{L}_{\diag}(X,\seq{e}{n})$ is isomorphic to $X^*$. It remains to
  show that the dual space $Z^*$ is isomorphic to a complemented
  subspace of $X^*$.

  We define $w_n=e_{2n-1}-e_{2n}$ for $n=1,2,\ldots$.
   \begin{claim}
   The basic sequence \seq{w}{n} of $X$ is 2-equivalent to the basis
  \seq{z}{n} of $Z$, i.e. for every sequence of scalars $(\lambda_k)_{k=1}^d$
  \begin{equation}\label{eq7}\|\sum\limits_{k=1}^d\lambda_kz_k\|_Z
   \le  \|\sum\limits_{k=1}^d\lambda_kw_k\|_X\le
    2  \|\sum\limits_{k=1}^d\lambda_kz_k\|_Z
  \end{equation}
   \end{claim}
     \begin{proof}[\bf Proof of the claim]
  Let $z=\sum\limits_{k=1}^d\lambda_kz_k$ be a finitely supported vector
  in $Z$. We choose a functional $z^*=\sum\limits_{k=1}^d a_kz_k^*\in B_{Z^*}$
  such that $z^*(z)=\|z\|_Z$. Then the functional
  $f=\sum\limits_{k=1}^da_ke_{2k-1}^*$ belongs to $K$ and thus
  \[\|\sum\limits_{k=1}^d\lambda_kw_k\|_X\ge f(\sum\limits_{k=1}^d\lambda_kw_k)
  =\sum\limits_{k=1}^d a_k\lambda_k=z^*(z)=
  \|\sum\limits_{k=1}^d\lambda_kz_k\|_Z.\]

 Next we prove the inequality in the right side of \eqref{eq7}.
  Let $f=\sum\limits_{i=1}^{d'} a_i\chi_{I_i}$ be an arbitrary
  functional in $K$.  Observe that if $\min I_i$ is odd and $\max
  I_i$ is even then $\chi_{I_i}(w_k)=0$ for all $k$.
  We set
  \begin{eqnarray*}
  A_{0}&=&\{i:\; \min I_i\mbox{ is even and }\max I_i
  \mbox{ is odd}\}\\
  A_{1}&=&\{i:\; \min I_i\mbox{ is even and }\max I_i
  \mbox{ is even}\}\\
  A_{2}&=&\{i:\; \min I_i\mbox{ is odd and }\max I_i
  \mbox{ is odd}\}
  \end{eqnarray*}
 For $i\in A_0\cup A_1$ let $\min I_i=2p_i$ and for
 $i\in A_0\cup A_2$ let $\max I_i=2q_i-1$. It follows that
 \begin{eqnarray*}
 f(\sum\limits_{k=1}^d\lambda_kw_k)& = &
   \sum\limits_{i\in A_0}a_i(-\lambda_{p_i}+\lambda_{q_i})
  +\sum\limits_{i\in A_1}a_i(-\lambda_{p_i})
  +\sum\limits_{i\in A_2}a_i \lambda_{q_i}\\
  & = & \sum\limits_{i\in A_0\cup A_1}(-a_i)\lambda_{p_i}+
        \sum\limits_{i\in A_0\cup A_2}a_i\lambda_{q_i}\\
 & = &
 (\sum\limits_{i\in A_0\cup A_1}(-a_i)z_{p_i}^*)(\sum\limits_{k=1}^d\lambda_kz_k)+
 (\sum\limits_{i\in A_0\cup A_2}
 a_iz_{q_i}^*)(\sum\limits_{k=1}^d\lambda_kz_k)\\
 & \le & 2\|\sum\limits_{k=1}^d\lambda_kz_k\|_Z
  \end{eqnarray*}
  where we have used the 1-unconditionality and the
  subsymmetricity of the basis and the fact that
  $\|\sum\limits_{i=1}^d a_iz_i^*\|\le 1$.
 Thus it follows that $\|\sum\limits_{k=1}^d\lambda_kw_k\|_X\le
              2\|\sum\limits_{k=1}^d\lambda_kz_k\|_Z$.
 \end{proof}
 It follows from the claim that the space $Z$ is isomorphic to the subspace
 $W=\overline{\spann}\{w_n:\;n\in\N\}$ of $X$.
 We claim that $W$ is a complemented subspace of $X$. Indeed,
 let $P:X\to W$ with
 $P(\sum\limits_{n=1}^{\infty}\lambda_ne_n)=\sum\limits_{n=1}^{\infty}\lambda_{2n-1}w_n$.
 Then, for any choice of scalars $(\lambda_n)_{n=1}^{2d}$ we have that
 \[ \|\sum\limits_{n=1}^{d}\lambda_{2n-1}w_n\|_X \le
 2\|\sum\limits_{n=1}^d\lambda_{2n-1}z_n\|_Z\le
 2\|\sum\limits_{n=1}^{2d}\lambda_ne_n\|_X.\]
 Thus $\|P\|\le 2$, while, since obviously $P(w_n)=w_n$ for all $n$,
 $P$ is a projection onto $W$.

 Therefore $Z^*$, being isomorphic to $W^*$, is isomorphic to a
 complemented subspace of $X^*$. Since $\mc{L}_{\diag}(X,\seq{e}{n})$ is isomorphic to
 $X^*$, we conclude that the space of diagonal operators $\mc{L}_{\diag}(X,\seq{e}{n})$ has a
 complemented subspace isomorphic to $Z^*$.
  \end{proof}

 \begin{theorem}
 There exists  a Banach space $X$ with a Schauder basis \seq{e}{n}
 such that its space of diagonal operators $\mc{L}_{\diag}(X)$ is
 nonseparable and does not contain $c_0(\N)$ or $\ell_1(\N)$.
 \end{theorem}
 \begin{proof}[\bf Proof.]
 The space $X=JT_I$ that we define, is a variant of the classical James Tree
 space $JT$. We consider the dyadic tree $(\mc{D},\preceq)$
 with its standard interpretation as the set of all finite
 sequences of 0's and 1's with  partial order
 defined by the relation $a\preceq \beta$ iff $a$ is an initial
 segment of $\beta$.
 The lexicographical order of $\mc{D}$ is defined through the
 bijection $h:\mc{D}\to\N$ defined by the rule $h(\emptyset)=1$,
 $h(\e_1\ldots\e_n)=2^n+\sum\limits_{j=1}^n\e_j2^{n-j}$.
 Thus the set $\N$ of the natural numbers is endowed with a
 partial order $\preceq$ such that $(\N,\preceq)$ coincides with
 the dyadic tree, while the natural order of $\N$
 coincides with the lexicographical order of $\mc{D}$.
 Thus a subset $A$ of $\N$ is called a segment of the dyadic tree
 if the set $h^{-1}(A)$ is a segment of $\mc{D}$.
 We set \begin{eqnarray*}
 \mathscr{S}& = & \{A:\;A \mbox{ is a finite interval of
 }\N\}\cup  \\& &
  \qquad\qquad \{A:\; A \mbox{ is a finite segment of the dyadic
 tree}\}.
 \end{eqnarray*}
 We observe that for $A,B\in\mathscr{S}$ we have that $A\cap B\in
 \mathscr{S}$.

 We define the  subset $K$ of $c_{00}(\N)$ as
 \[K=\big\{\sum\limits_{i=1}^d a_i\chi_{A_i}:\;
 (A_i)_{i=1}^d\mbox{ are pairwise disjoint members of
 }\mathscr{S},\;\sum\limits_{i=1}^d a_i^2\le 1,\;d\in\N\big\}. \]
 The space $JT_I$ is the completion of the normed space
 $(c_{00}(\N),\|\cdot\|_{K})$.
 The standard Hamel
 basis \seq{e}{n} of $c_{00}(\N)$ is a normalized bimonotone Schauder basis of $JT_I$.
 Since clearly $\pm\chi_{\{1,\ldots,n\}}\in K$ for all $n$,
 condition (3)(a) of Theorem  \ref{th1} is satisfied  with constant $C_1=1$.
 Also, if $f=\sum\limits_{i=1}^da_i\chi_{A_i}$ and
 $g=\sum\limits_{j=1}^{d'}\beta_j\chi_{B_j}$ belong to $K$, then we have that
 $f\cdot g=
 \sum\limits_{i=1}^d\sum\limits_{j=1}^{d'}a_i\beta_j\chi_{A_i\cap
 B_j}$ with $A_i\cap B_j\in\mathscr{S}$ and
 $\sum\limits_{i=1}^d\sum\limits_{j=1}^{d'}
 a_i^2\beta_j^2\le1$, hence $f\cdot g\in K$. Thus
 condition (3)(b) of Theorem  \ref{th1} is satisfied  with constant $C_2=1$.
 Therefore Theorem \ref{th1}
  entails that the algebra $\mc{L}_{\diag}(JT_I)$, of all diagonal operators of the space $JT_I$
  with respect to the  basis \seq{e}{n}, is isometric to
  $JT_I^*$.
 Similarly to the proof for the classical James Tree space $JT$ (see e.g. \cite{J}, \cite{LS}),
 it is proved that $JT_I$ contains no isomorphic copy of $\ell_1(\N)$,
 hence $JT_I^*$ contains no isomorphic copy of $c_0(\N)$, while
 $JT_I^*$ is nonseparable and
 $(JT_I)^{**}$ is isomorphic to $JT_I \oplus \ell_2( c)$.
 The later implies that  $JT_I^*$ does not contain $\ell_1(\N)$.
 Therefore
 the  space of diagonal operators $\mc{L}_{\diag}(JT_I)$
 is nonseparable and does not contain $c_0(\N)$ or $\ell_1(\N)$.
  \end{proof}

\section{Definition of the space $\eqs_D$}
 The content of the present section is the definition of the space $\eqs_D$ that we shall
 study in the subsequent sections. The novelty of the definition of
 the space $\eqs_D$ (compared to earlier HI constructions)
  is that, in each inductive  step of the definition of its norming
  set $D$,
 we close under pointwise products, in order to obtain a set $D$
 satisfying condition (3)(b) of Theorem \ref{th1}. This forces the dual space $\eqs_D^*$
 to be a Banach  algebra. We also  include in
 the norming set $D$ of the space $\eqs_D$, the functionals $\pm\chi_I$ for all finite intervals $I$,
 in order to satisfy condition (3)(a) of Theorem \ref{th1}.

  \begin{definition}[The space $\eqs_D$]\label{D13}
 We  fix two sequences of integers \seq{m}{j}, \seq{n}{j},
 as follows:
 \begin{itemize}
 \item     $m_1=2\;\;$  and $\;\;m_{j}=m_{j-1}^5$ for $j>1$.
 \item    $n_1=4\;\;$ and
 $n_{j}>m_{j}^{4\log_2(m_j)+2}\cdot Q_{j}^{2\log_2(m_j)}\qquad$ for $j>1,$ \\where
 $\;\;Q_{j}= \big(n_{j-1}\cdot \log_2(m_{j})\big)^{\log_2(m_{j})}$
 \end{itemize}

 Let
  $\Q_s$ denote the set of all finite
 sequences $(\phi_1,\phi_2,\ldots,\phi_d)$ such that
 $\phi_i\in \co$, $\phi_i\neq 0$ with $\phi_i(n)\in \Q$ for all $i,n$ and
 $\phi_1<\phi_2<\cdots<\phi_d$.
 We fix a pair $\Omega_1,\Omega_2$ of disjoint infinite subsets of
 $\N$.
 From the fact that $\Q_s$ is
 countable we are able to define a Gowers-Maurey type injective
 coding function
 $\sigma:\Q_s\to \{2j:\;j\in\Omega_2\}$ such that
 $m_{\sigma(\phi_1,\phi_2,\ldots,\phi_d)}>
 \max\{\frac{1}{|\phi_i(e_l)|}:\;l\in\supp
 \phi_i,\;i=1,\ldots,d\}\cdot\max\supp \phi_d$.

 We shall inductively define a triple $(K_n,M_n,D_n)$ of subsets of $c_{00}(\N)$,
  $n=0,1,2,\ldots$
 and $K_n^j$, $j=0,1,2,\ldots$ with
 $K_n=\bigcup\limits_{j=0}^{\infty}K_n^j$.

 We first define $G=\{\pm\chi_I:\; I \mbox{ is a finite interval of
 }\N\}$.
 We set $K_0^0=G$, $K_0^j=\emptyset$, $j=1,2,\ldots$ and  $K_0=G$,
 $M_0=G$, $D_0=\conv_{\Q}(G)$.

 Suppose that the sets $(K_n^j)_{j=0}^{\infty}$, $M_n$, $D_n$ have been
 defined. We set $K_{n+1}^0=G$ and for $j=1,2,\ldots$ we define
 \[K_{n+1}^{2j}=\Big\{\frac{1}{m_{2j}}\sum\limits_{i=1}^d f_i:\;
 f_i\in D_n,\; f_1<\cdots<f_d,\; d\le n_{2j}\Big\}\cup K_n^{2j}\]
 and
 \begin{eqnarray*}
 K_{n+1}^{2j-1} &= &\Big\{\pm E(\frac{1}{m_{2j-1}}
 \sum\limits_{i=1}^{n_{2j-1}} f_i):\;
 f_1 \in K_n^{2j_1} \mbox{ for some }j_1\in\Omega_1   \\
   & & \mbox{ with }
  m_{2j_1}^{1/2}>n_{2j-1},
  \; f_{i+1}\in K_n^{\sigma(f_1,\ldots,f_i)},\\
  & &
  f_1<f_2<\cdots<f_{n_{2j-1}}, \\
  & &
  E\; \mbox{ is an interval of }\N
  \Big\}\cup K_n^{2j-1}.
  \end{eqnarray*}
 We set
 \begin{eqnarray*}
 K_{n+1} &= &\bigcup\limits_{j=0}^{\infty}K_{n+1}^j,\\
 M_{n+1} &= & \{f_1\cdot \;\ldots\; \cdot f_d:\; f_i \in
 K_{n+1},\;i=1,\ldots,d,\; d\in\N\}\\
 D_{n+1} &=& \conv_{\Q}(M_{n+1}).
  \end{eqnarray*}

 The inductive construction has been completed. We finally set
 $K^{j}=\bigcup\limits_{n=0}^{\infty}K_n^j$ for $j=1,2\ldots$ and
 $K=\bigcup\limits_{n=0}^{\infty}K_n$,
 $M=\bigcup\limits_{n=0}^{\infty}M_n$,
 $D=\bigcup\limits_{n=0}^{\infty}D_n$.

 The Banach space $\eqs_D$ is the completion of the normed space
 $(c_{00}(\N),\|\cdot\|_D)$.
 \end{definition}

 \begin{remark}\label{R23}
  The norming set $D$ of the space $\eqs_D$ is closed under  pointwise products.
 Indeed, since $K_n\subset K_{n+1}$ and $K_n\subset M_n\subset D_n$ for all $n$, in order
  to show the former it
  is enough to show that
  each $D_n$ is closed under  pointwise products.
  Let $f,g\in D_n$. Then $f=\sum\limits_{i=1}^d\lambda_i f_i$,
  $g=\sum\limits_{j=1}^{d'}\mu_jg_j$ as convex combinations, with $(f_i)_{i=1}^d$ and $(g_j)_{j=1}^{d'}$
  in $M_n$. The pointwise product $f\cdot g$ takes the form
  $f\cdot g=\sum\limits_{i=1}^d\sum\limits_{j=1}^{d'}\lambda_i \mu_j \;f_i\cdot g_j$ as a convex combination of
  the family of functionals $(f_i\cdot g_j)_{i=1,j=1}^{d,\quad d'}$ with each $f_i\cdot g_j$ belonging
  to $M_n$. Therefore $f\cdot g\in D_n\subset D$.

   It follows that  $\|f\cdot g\|\le \|f\|\cdot\|g\|$  for every $f,g\in \eqs_D^*$,
   therefore the space $\eqs_D^*$ is a Banach algebra.
 \end{remark}

 \begin{remark}\label{R6}
 The set $D$ is the minimal subset of $c_{00}(\N)$ satisfying the
 following properties:
 \begin{enumerate}
 \item[(i)] $G\subset D$, i.e. the set $D$ contains $\pm\chi_I$
  for any finite interval $I$ of $\N$.
 \item[(ii)] $D$ is closed under the
 $(\mc{A}_{n_{2j}},\frac{1}{m_{2j}})_j$
 operations.
 \item[(iii)] For each $j$, the set  $D$ is closed under the
 $(\mc{A}_{n_{2j-1}},\frac{1}{m_{2j-1}})$ operation on $2j-1$
 special sequences (see  Definition \ref{D8} below).
 \item[(iv)] $D$ is closed  under the restriction of its
 elements to  intervals of $\N.$
 \item[(v)] $D$ is rationally convex.
 \item[(vi)] $D$ is closed under  pointwise products.
 \end{enumerate}
 \end{remark}

  \begin{definition}\label{D8}
  A  block sequence $(f_i)_{i=1}^{n_{2j-1}}$ is said to be
  a $2j-1$ special sequence if $f_1\in K^{2j_1}$ for some $j_1\in
  \Omega_1$ with $m_{2j_1}^{1/2}>n_{2j-1}$ and $f_{i+1}\in
  K^{\sigma(f_1,\ldots,f_i)}$ for $1\le i<n_{2j-1}$.
  \end{definition}

  \begin{remark}\label{R3} The
  sequence \seq{e}{n} is clearly a normalized bimonotone
 Schauder basis of the space $\eqs_D$.
  From the fact that the
  norming set $D$ is closed under the $(\mc{A}_{n_{2j}},\frac{1}{m_{2j}})_j$
  operations, and taking into account that $\lim\limits_j \frac{m_j}{n_j}=0$,
   it also follows that the basis  \seq{e}{n} is  boundedly complete.
   Hence the space
 $(\eqs_D)_*=\overline{\spann}\{e_n^*:\;n\in\N\}$ is a predual of the
 space $\eqs_D$. Notice also that, as a consequence of the fact
 that the norming set $D$ is rationally convex, the set $D$
   is pointwise dense
 in the unit ball $\mbox{B}_{\eqs_D^*}$ of the dual space. Since the
 set $D$ is closed under the
 $(\mc{A}_{n_{2j}},\frac{1}{m_{2j}})_{j\in\N}$ operations, we
 get that the  unit ball $\mbox{B}_{\eqs_D^*}$ shares the same
 property, i.e.  if $j\in\N$, $d\le n_{2j}$ and $f_1<f_2<\cdots<
 f_d$ with $\|f_i\|\le 1$ then
 $\|\frac{1}{m_{2j}}\sum\limits_{i=1}^df_i\|\le 1$.
 \end{remark}

 \begin{remark}
 Each $f\in D$ has one of the following forms:
 \begin{enumerate}
 \item[(a)] $f=\pm \chi_I$, $I$ an interval of $\N$
 (i.e. $f\in G$). We set $w(f)=1$.
 \item[(b)] $f=\frac{1}{m_{j}}\sum\limits_{i=1}^d f_i$ with
 $f_1<\cdots <f_d$, $d\le n_j$, $f_i\in D$ (i.e. $f\in K^j\subset K$).
 In this case we set  $w(f)=m_j$.
 \item[(c)] The functional $f\in M$ is the pointwise product
  $f=f_1\cdot\ldots\cdot f_d$ with each
 $f_i\in K$. In this case we define
 $w(f)=w(f_1)\cdot\ldots\cdot w(f_d)$.
 \item[(d)] $f$ is a rational convex combination
 $f=\sum\limits_{i=1}^d \lambda_if_i$ with each $f_i$ belonging
 to cases (a), (b), (c).
 \end{enumerate}
 \end{remark}

  We next show, that a functional $f\in M$
  (i.e. a pointwise product), can be written as $f=\frac{1}{w(f)}\sum h_i$
  with $(h_i)_i$ being a sequence of successive functionals belonging to the norming set
  $D$, of length determined by $w(f)$.

  \begin{lemma}\label{L20}
 If $(I_i)_{i=1}^d$, $(J_j)_{j=1}^{d'}$ is any pair of families of
 successive intervals of $\N$ (i.e. $I_1<\cdots<I_d$ and
 $J_1<\cdots<J_{d'}$) then the cardinality of the nonempty sets of
 the family $\{I_i\cap J_j,\; 1\le i\le d,\; 1\le j\le d'\}$ is at
 most $d+d'-1$.
 \end{lemma}
 \begin{proof}[\bf Proof.]
 We proceed by induction on the sum $d+d'$. For $d+d'=2$, i.e. if
 $d=d'=1$ there is nothing to be proved. Let  $k\ge 2$ and suppose
 that the lemma is true for $d+d'\le k$. We prove the result for
 $d+d'=k+1$. If $d=1$ or $d'=1$ then the result is straightforward, so
 we assume that $d>1$ and $d'>1$. Without loss of generality we
 may assume that $\min J_{d'}\le \min I_d$. From our inductive
 assumption the family $\{I_i\cap J_j,\; 1\le i \le d-1,\; 1\le j\le
 d'\}$ has at most $(d-1)+d'-1$ nonempty sets. Since
 $\min J_{d'}\le \min I_d$ we get that $I_d\cap J_j=\emptyset$ for
 $j=1,\ldots,d'-1$. Thus the only set between $J_1,\ldots , J_{d'}$
 that may
 intersect  $I_d$ is $J_{d'}$. Therefore the nonempty sets of the
 family $\{I_i\cap J_j,\; 1\le i\le d,\; 1\le j\le d'\}$ are at
 most $[(d-1)+d'-1]+1$ i.e. at most $d+d'-1$.
 \end{proof}

 \begin{proposition}\label{C1}
 Let $f\in M$, $f=f_1\cdot\ldots\cdot f_r$ with $f_i\in K$,
 $w(f_i)=m_{j_i}$, $i=1,\ldots,r$. Then the functional $f$ takes
 the form $f=\frac{1}{w(f)}\sum\limits_{i=1}^d h_i$ where $h_i\in
 D$, $h_1<\cdots<h_d$ and $d\le n_{j_1}+\cdots+n_{j_r}-(r-1)$.
 Moreover, if $f\in M_{n+1}$, then we may select each $h_i$ to belong to $D_n$.
 \end{proposition}
 \begin{proof}[\bf Proof.]
 Let $f\in M_{n+1}$. Then
 $f=f_1\cdot\ldots\cdot   f_{r}$ with each $f_i\in K_{n+1}$. We shall prove,
 by induction on $k$, that each product $f_1\cdot\ldots\cdot   f_{k}$, for $1\le k\le r$,
 takes the desired form. For $k=1$ there is nothing to
 be proved. Let $k<r$ and suppose  that
 $f_1\cdot\ldots\cdot f_k=\frac{1}{w(f_1)\cdot\ldots\cdot w(f_k)}
 \sum\limits_{l=1}^{d'} H_l$ with $H_l\in D_n$,
 $H_1<\cdots<H_{d'}$ and $d'\le n_{j_1}+\cdots+n_{j_k}-(k-1)$.
  Let also $f_{k+1}=\frac{1}{m_{j_{k+1}}}(f^{k+1}_1+\cdots+f^{k+1}_m)$,
 $m\le n_{j_{k+1}}$, with each $f^{k+1}_j\in D_n$.
 Applying  Lemma \ref{L20} to the families
 $(\ran H_l)_{l=1}^{d'}$ and  $(\ran f^{k+1}_j)_{j=1}^{m}$ we get
 that $\ran H_l\cap \ran f^{k+1}_j\neq \emptyset$ for at most
 $n_{j_1}+\cdots+n_{j_{k+1}}-k$ pairs $(l,j)$.
  Taking into account
 that the  set $D_n$ is closed under  pointwise
 products (see Remark \ref{R23}) we get that
 \[f_1\cdot\ldots\cdot  f_k\cdot  f_{k+1}=\frac{1}{w(f_1)\cdot\ldots\cdot
 w(f_k)} \frac{1}{m_{j_{k+1}}}(\sum\limits_{l=1}^{d'}H_l)
 (\sum\limits_{j=1}^{m}f^{k+1}_j)
 =\frac{1}{w}\sum\limits_{i=1}^d h_i\]
 where $w=w(f_1)\cdot \ldots \cdot w(f_{k})\cdot m_{j_{k+1}}=w(f_1\cdot\ldots \cdot f_k\cdot f_{k+1})$,
   $h_1<\cdots<h_d$ with each $h_i\in D_n$,
  and $d\le n_{j_1}+\cdots+n_{j_{k+1}}-k$.
 This completes the proof of the inductive step and the proof of the proposition.
 \end{proof}

 \begin{corollary}\label{C2}
 Let $f\in M$  with $w(f)<m_j$.
  Then the functional $f$ can be written in the form
 $f=\frac{1}{w(f)}\sum\limits_{i=1}^d h_i$ with $h_i\in
 D$, $h_1<\cdots<h_d$ and $d< n_{j-1}\log_2(m_j)$.
 \end{corollary}
 \begin{proof}[\bf Proof.]
 Let $f=f_1\cdot\ldots\cdot f_k$ with $f_i\in K$,
 $w(f_i)=m_{j_i}$, $i=1,\ldots,k$. Then $m_j>w(f)=w(f_1)\cdot
 \ldots\cdot w(f_k)\ge 2^k$ and hence $k<\log_2(m_j)$.

  Since $j_i\le j-1$ for each $i$, from Proposition \ref{C1} the
  functional $f$ takes the form
  $f=\frac{1}{w(f)}\sum\limits_{i=1}^d h_i$
  with $d\le n_{j_1}+\ldots+n_{j_k}-(k-1)<n_{j-1}\log_2(m_j)$.
  \end{proof}

  \begin{corollary}\label{C5}
 Let $f\in M$,  $f=f_1\cdot\ldots\cdot f_k$ with $f_i\in K$,
 such that   $w(f_i)<m_j$ for $i=1,\ldots,k$ and $w(f)<m_j^2$.
  Then the functional $f$ takes the form
 $f=\frac{1}{w(f)}\sum\limits_{i=1}^d h_i$ with $h_i\in
 D$, $h_1<\cdots<h_d$ and $d< n_{j-1}\log_2(m_j^2)$.
 \end{corollary}
 \begin{proof}[\bf Proof.]
 The proof is almost identical to that of Corollary \ref{C2}, so we omit it.
 %
  \end{proof}

 \begin{definition}\label{D4}
For  $f\in D$ we call tree of $f$ (or tree corresponding to the
analysis of $f$) a family of functionals $T_f=(f_a)_{a\in \mc{A}}$
indexed by  a finite tree $\mc{A}$, with each $f_a\in D$, such
that the following conditions are fulfilled:

 \begin{enumerate}
 \item[(i)] The tree $\mc{A}$ has a unique root $0\in\mc{A}$ and
 $f_0=f$.
 \item[(ii)] If $a$ is a maximal element of the tree $\mc{A}$ then
 $f_a\in G$. In this case we say that $f_a$ is of type 0 with
 weight  $w(f_a)=1$.
 \item[(iii)] For every non-maximal $a\in\mc{A}$, denoting by
 $S_a$ the set of immediate successors of $a$ in the tree
 $\mc{A}$, $S_a=\{\beta_1,\cdots,\beta_d\}$, one of the following
 holds:
 \begin{enumerate}
 \item[(a)] $f_{\beta_1}<\cdots<f_{\beta_d}$
 and  $f=\frac{1}{w(f_a)}\sum\limits_{i=1}^df_{\beta_i}$
 where $w(f_a)=m_{j_1}\cdot \ldots\cdot m_{j_r}$ and
 $d\le n_{j_1}+\cdots+n_{j_r}$. In this case we say that
 $f_a$ is of type I with weight $w(f_a)$.
 \item[(b)]  There exists a family $(\lambda_{\beta_i})_{i=1}^d$ of
 positive rationals with $\sum\limits_{i=1}^d\lambda_{\beta_i}=1$ such
 that $f_a=\sum\limits_{i=1}^d\lambda_{\beta_i}f_{\beta_i}$ and for each
 $i$, $\ran f_{\beta_i}\subset \ran f_a$ and $f_{\beta_i}$
 is either of type I or of type 0.
 In this case we say that $f_a$ is of type II.
\end{enumerate}
 \end{enumerate}
 \end{definition}

 \begin{remark}\label{R1}
  Every $f\in D$ admits a tree (not necessarily unique). Indeed, it can be shown
  that each $f\in D_n$ admits a tree, using induction on $n$ and applying Proposition
  \ref{C1} in each inductive step.
  \end{remark}

 \begin{proposition}\label{P6}
 The Banach algebra $\mc{L}_{\diag}(\eqs_D)$ of diagonal operators
 of $\eqs_D$ with respect to the basis $\seq{e}{n}$, is isometric
 to the dual space $\eqs_D^*$.
 \end{proposition}
 \begin{proof}[\bf Proof.]
  Since the norming set $D$ of the space $\eqs_D$ contains the
 characteristic functions $\pm\chi_{\{1,\ldots,n\}}$ for every $n$
 (this is the reason we have included the set $G$ in the norming
 set $D$) i.e. $\pm \sum\limits_{i=1}^ne_i^*\in D$
 for every $n$, condition (3)(a) of  Theorem \ref{th1}
 is satisfied with constant $C_1=1$.
  The norming set
 $D$ is closed under  pointwise products, i.e.
  $D\cdot D\subset D$,
 hence condition (3)(b) of  Theorem \ref{th1}
 is also satisfied with constant $C_2=1$.
 Theorem \ref{th1} entails that the space of diagonal operators
 $\mc{L}_{\diag}(\eqs_D)$ is isometric to  the dual space $\eqs_D^*$.
 \end{proof}

 \section{The basic inequality and exact pairs}
This section is mainly devoted to the statement and the proof of
the  basic inequality (Proposition \ref{P1}). For this we follow
the standard method which has been used in earlier works (i.e.
\cite{AAT}, \cite{ADT}, \cite{ALT2}, \cite{AM}, \cite{AT2}). The
general scheme for this method goes as follows. We first define an
auxiliary space which is a mixed Tsirelson space and the basic
inequality shows that the action of every $f\in D$ on an average
of a RIS is dominated by the action of a functional $g$ on the
corresponding average of the basis of the auxiliary space plus a
small error. This provides, among others, the upper estimate of
the norm of  averages of a RIS.  We  show that the space $\eqs_D$
is of codimension 1 in $\eqs_D^{**}$ hence is quasireflexive.
Finally, we define the exact pairs, a key ingredient for the
definition of dependent sequences.

 \begin{definition}\label{D6}
 Let $C\ge 1$, $\e>0$. A block sequence
 $(x_k)_{k}$ in $\eqs_D$ is said to be a $(C,\e)$
 Rapidly Increasing Sequence (RIS) if there exists
 a strictly increasing sequence $(j_k)_k$ of integers
  such that the
following conditions are satisfied:
 \begin{enumerate}
  \item[(i)] $\|x_k\|\le C$ and $\|x_k\|_G\le \e$ for each $k$.
  \item[(ii)] $\frac{1}{m_{j_1}}\le \e$
   and $\#\supp(x_k)\cdot \frac{1}{m_{j_{k+1}}}\le \e$ for all
   $k$.
   \item[(iii)] For every $k$ and $f\in D$ with $w(f)<m_{j_k}$ it holds
 that $|f(x_k)|\le \frac{C}{w(f)}$.
 \end{enumerate}
 We call the sequence of integers  $(j_k)_k$, the associated sequence of the RIS $(x_k)_k$.
 \end{definition}

 \begin{definition}[The auxiliary space]
 Let $W$ be the minimal subset of $c_{00}(\N)$ such that
 \begin{enumerate}
 \item[(i)] It contains $\pm e_n^*$, $n\in\N$.
 \item[(ii)] It is closed under the $(\mc{A}_{2n_j},\frac{1}{m_j})$
  operation for every $j$.
 \item[(iii)] It is rationally convex.
 \end{enumerate}
 We also define $W'$ as the minimal subset of $c_{00}(\N)$
 satisfying the above conditions (i), (ii).
 \end{definition}

 \begin{remark}\label{R4}
 It is easily seen that a subset of $c_{00}(\N)$ which is closed
 under the $(\mc{A}_n,\theta)$ and  $(\mc{A}_{n'},\theta')$
 operations, it is also closed under the the
 $(\mc{A}_{nn'},\theta\theta')$ operation. It follows that the set
 $W$ is closed under the
 $(\mc{A}_{(2n_{j_1})\cdot\ldots\cdot(2n_{j_k})}
 ,\frac{1}{m_{j_1}\cdot\ldots\cdot m_{j_k}})$
  operation, for every $j_1,\ldots,j_k\in\N$ (not
 necessarily distinct). Since
 $\sum\limits_{i=1}^k2n_{j_i}\le\prod\limits_{i=1}^k2n_{j_i}$ we get
 that the set $W$ (and  the set $W'$ also) is closed under the
 $(\mc{A}_{2n_{j_1}+\cdots+2n_{j_k}},\frac{1}{m_{j_1}\cdot\ldots\cdot
 m_{j_k}})$ operation.
 \end{remark}

  \begin{remark}\label{remsam1}
 The trees for functionals $g\in W$, are defined in a similar
 manner as the corresponding ones for $g\in D$ (Definition \ref{D4}),
 the only difference being that the
 functionals corresponding to maximal elements are of the form
 $\pm e_r^*$. For $f\in W'$ the trees are defined as those for
 $g\in W$, the only difference being that we
 require that no functionals of type II appear.
 \end{remark}

 \begin{lemma}\label{L22}
 Let $g\in W$ with a tree $(g_a)_{a\in \mc{A}}$.  Then the
 functional $g$ is  a rational convex combination $g=\sum\limits_{i\in
 I}\lambda^ig^i$,
   such that for each $i\in I$,  $g^i\in
 W'$,
 the functional $g^i$ has a tree $(g^i_a)_{a\in\mc{A}^i}$
 and there exists an order preserving map $\Phi^i:\mc{A}^i\to
 \mc{A}$ satisfying the following.
 \begin{enumerate}
  \item[(i)] For every maximal node $a\in\mc{A}^i$, $\Phi^i(a)$ is
  a maximal node of $\mc{A}$ and
 $g^i_a=g_{\Phi^i(a)}$.
 \item[(ii)] For every  non-maximal $a\in\mc{A}^i$, the
 functionals
  $g^i_a$, $g_{\Phi^i(a)}$ are of type I with
  $w(g^i_a)=w(g_{\Phi^i(a)})$
  and  $\# S^i_a=\# S_{\Phi^i(a)}$, where
  $S^i_a$ denotes the set of immediate
 successors of  $a\in\mc{A}^i$ and  $S_\gamma$ is the set of immediate
 successors of a $\gamma\in \mc{A}$.
 \item[(iii)] If the functional $g$ is weighted then $w(g^i)=w(g)$.
 \end{enumerate}
 \end{lemma}
 \begin{proof}[\bf Proof.]
 We shall prove, using backward induction, that for each
 $a\in \mc{A}$ the functional $g_a$ with the  tree $(g_a)_{a\in\mc{A}^{\succeq \beta}}$,
 where $\mc{A}^{\succeq \beta}=\{a\in\mc{A}:\; a \succeq \beta\}$,
 satisfies the conclusion of the lemma.
 This will finish the proof of the lemma, since
  $g=g_0$, where
 $0\in\mc{A}$ denotes  the unique root of the tree $\mc{A}$.

 The first inductive step concerns $a\in\mc{A}$ which is maximal.
 In this case, setting $I_a=\{1\}$, $\lambda^1=1$,
 $\mc{A}^1=\{a\}$ and $\Phi^1(a)=a$, the required conditions are
 obviously satisfied.

 Let us pass to the general inductive step. We consider
 $a\in\mc{A}$ which is non-maximal,
 $S_a=\{\beta_1,\ldots,\beta_d\}$ and we assume that for each
 $k=1,\ldots,d$, the functional $g_{\beta_k}$ takes the form
 $g_{\beta_k}=\sum\limits_{i\in
 I_{\beta_k}}\lambda_{\beta_k}^ig_{\beta_k}^i$ and each $g_{\beta_k}^i\in W'$
 has a tree $(g_{\beta_k,\gamma}^i)_{\gamma\in \mc{A}_{\beta_k}^i}$
 such that there exists an order preserving map
 $\Phi_{\beta_k}^i:\mc{A}_{\beta_k}^i\to  \mc{A}^{\succeq \beta_k}$
 satisfying conditions (i), (ii), (iii).
 We distinguish two cases.

 {\bf Case 1.} The functional $g_a$ is of type I,
 $g_a=\frac{1}{w(g_a)}(g_{\beta_1}+\cdots+g_{\beta_d})$.\\
 We set $I_a=I_{\beta_1}\times\cdots \times I_{\beta_d}$ and for each
 $i=(i_1,\ldots,i_d)\in I_a$ we define $\lambda_a^i=
 \lambda_{\beta_1}^{i_1}\cdot\ldots\cdot \lambda_{\beta_d}^{i_d}$
 and
 $g_a^i=\frac{1}{w(g_a)}(g_{\beta_1}^{i_1}+\cdots+g_{\beta_d}^{i_d})$.
 It is evident that $g_a$ equals to the convex combination
 $\sum\limits_{i\in I_a}\lambda_a^i g_a^i$.
 For each $i= (i_1,\ldots,i_d)\in I_a$ we define the tree $\mc{A}^i_a$ as the disjoint union
  $\mc{A}_a^i=\{a\}\cup\bigcup\limits_{k=1}^d \mc{A}_{\beta_k}^{i_k}$
 with its ordering defined by the rule $\delta_1\preceq \delta_2$
 if and only if $\delta_1=a$ or if there exists
 $k\in\{1,\ldots,d\}$ such that $\delta_1,\delta_2\in \mc{A}_{\beta_k}^{i_k}$
 and $\delta_1\preceq \delta_2$ in the ordering of
 $\mc{A}_{\beta_k}^{i_k}$. We also define the order preserving map
 $\Phi_a^i:\mc{A}_a^i\to \mc{A}^{\succeq a}$ as follows;
 $\Phi_a^i(a)=a$ and
 $\Phi_a^i(\gamma)=\Phi_{\beta_k}^{i_k}(\gamma)$ for
 $\gamma\in\mc{A}_{\beta_k}^{i_k}$. It is trivial to observe that
 conditions (i), (ii), (iii) are satisfied.

 {\bf Case 2.}  The functional $g_a$ is of type II,
 $g_a=\sum\limits_{k=1}^d \lambda_{\beta_k}g_{\beta_k}$.\\
 Then, using our inductive hypothesis, we get that
 $g_a=\sum\limits_{k=1}^d \sum\limits_{i\in
 I_{\beta_k}}(\lambda_{\beta_k}\lambda_{\beta_k}^i)g_{\beta_k}^i$
 as a convex combination, while the required conditions
 (i), (ii), (iii) are obviously satisfied.

 The proof of the Lemma is complete.
 \end{proof}

 \begin{lemma}\label{L21}
 Let $g\in W$, $g=\frac{1}{w(g)}\sum\limits_{i=1}^dg_i$ where
 $w(g)=m_{j_1}\cdot\ldots\cdot m_{j_r}$ and $g_1<\cdots<g_d$ with
 $g_i\in W$ and $d\le 2n_{j_1}+\cdots+2n_{j_r}$. Let also $j\in
 \N$. Then
 \[   |g(\frac{1}{n_j}\sum\limits_{k=1}^{n_j}e_k)|\le
     \left\{ \begin{array} {l@{\quad} l}
           \frac{ 2}{w(g)m_{j}} & \mbox{ if }w(g)<m_{j} \\[4mm]
   \frac{1}{w(g)} & \mbox{ if }w(g)\ge
   m_{j}.
   \end{array}\right.     \]
 \end{lemma}
 \begin{proof}[\bf Proof.]
 From Lemma \ref{L22}, we may  assume that
  $g\in W'$ and $g_i\in W'$ for each $i=1,\ldots,d$.
 We start with the following claim.
 \begin{claim} For every $f\in W'$,
 $\#\{k:\;|f(e_k)|>\frac{1}{m_j}\}\le(2n_{j-1})^{\log_2(m_j)-1}$.
\end{claim}
 \begin{proof}[\bf Proof of the claim.] We may select a family of
 functionals  $(f_a)_{a\in\mc{A}}$ indexed by a finite tree
 $\mc{A}$, with $f_a\in W$, such that
 \begin{enumerate}
 \item[(i)] The tree $\mc{A}$ has a unique root $0\in\mc{A}$ and
 $f_0=f$.
 \item[(ii)] For every maximal node $a\in\mc{A}$, $f_a=\pm e_n^*$.
 \item[(iii)] For every non-maximal node $a\in\mc{A}$,
  there exists
 $j_a\in\N$ such that
 $f_a=\frac{1}{m_{j_a}}(f_{\beta_1}+\cdots+f_{\beta_d})$ and $d\le
 2n_{j_a}$ where $S_a=\{\beta_1,\ldots,\beta_d\}$ is  the set of
 immediate successors
 of $a$ in $\mc{A}$ and $f_{\beta_1}<\cdots<f_{\beta_d}$.
 \end{enumerate}
 We may assume that $f(e_k)>\frac{1}{m_j}$ for every
 $k\in\supp(f)$. Since $w(f_a)\ge m_1\ge 2$ for every non-maximal
 $a\in\mc{A}$ it follows that the cardinality of every branch of
 the tree $\mc{A}$ is less than $\log_2(m_j)$. An easy inductive
 argument implies that $\# \supp(f)\le(2n_{j-1})^{\log_2(m_j)-1}$.
 \end{proof}

 Let now $g\in W'$.
 The case $w(g)\ge m_j$ is obvious.  Assume that $w(g)<m_j$.
 Then, as in the proof of  Corollary \ref{C2}, it follows that
 $d\le 2n_{j-1}\log_2(m_j)$.
 For $i=1,\ldots,d$, set $L_i=\{k:\; |g_i(e_k)|>\frac{1}{m_j}\}$ and
 $L=\bigcup\limits_{i=1}^dL_i$. From the claim above we get that
 $\# L_i\le (2n_{j-1})^{\log_2(m_j)-1}$ for each $i$, thus
 $\# L\le (2n_{j-1})^{\log_2(m_j)}\log_2(m_j)$. Therefore,
  splitting the functional $g$ as $g=g|_{L} + g|_{{\N\setminus L}}$ we get that
\begin{eqnarray*}
  |g(\frac{1}{n_j}\sum\limits_{k=1}^{n_j}e_k)| &
  \le &
 \frac{1}{w(g)}\frac{1}{n_j}\#L+\frac{1}{w(g)}\frac{1}{m_j}  \\
 & \le &
 \frac{1}{w(g)}\big(\frac{1}{n_j}(2n_{j-1})^{\log_2(m_j)}\log_2(m_j)
 +\frac{1}{m_j}\big)\le
 \frac{2}{w(g)\cdot m_j}.
 \end{eqnarray*}
 \end{proof}

 \begin{lemma}\label{L3}
 Let $g\in W$ and suppose that the functional $g$ admits a tree
 $(g_a)_{a\in\mc{A}}$ with the following property.
 For every $a\in\mc{A}$ such that $g_a$ is of type I with $w(g_a)<m_{j_0}^2$,
 the cardinality of the set $S_a$ of immediate successors of $a$ in $\mc{A}$
 satisfies $\# S_a\le m_{j_0}^2 Q_{j_0}$ \big(recall from Definition \ref{D13} that
  $Q_{j_0}= \big(n_{j_0-1}\cdot \log_2(m_{j_0})\big)^{\log_2(m_{j_0})}$\big).
   Then
 \[   |g(\frac{1}{n_{j_0}}\sum\limits_{k=1}^{n_{j_0}}e_k)|\le \frac{2}{m_{j_0}^2}.\]
 \end{lemma}
 \begin{proof}[\bf Proof.]
  It is enough to prove the statement for $g\in W'$, since from Lemma \ref{L22}
   the functional $g$
  takes the form  $g=\sum\limits_{i\in I}\lambda^ig^i$ as a convex combination, with each
  $g^i\in W'$ and such that each $g^i$ satisfies the assumption of the statement.

  Let $g\in W'$ satisfying the  assumption of the statement of the lemma.
  We set
  \[
  B_1=\{k:\; |g(e_k)|>\frac{1}{m_{j_0}^2}\}\qquad
  B_2=\{k:\; |g(e_k)|\le\frac{1}{m_{j_0}^2}\}
  \]
  and we consider the functionals $g_1=g_{|B_1}$ and $g_2=g_{|B_2}$.
 Then obviously
  \begin{equation} \label{sam1}
  |g_2(\frac{1}{n_{j_0}}\sum\limits_{k=1}^{n_{j_0}}e_k)|\le \frac{1}{m_{j_0}^2}.
  \end{equation}
  The property of the statement remains valid for the functional $g_1$; let
  $(g_a)_{a\in\mc{A}}$ be a tree of the functional $g_1$,
  satisfying the aforementioned property. Note that no functionals of type II appear in this tree,
  since $g_1\in W'$ (see Remark \ref{remsam1}).
  The fact that $|g_1(e_k)|>\frac{1}{m_{j_0}^2}$ for every $k\in\supp g_1$ implies that every
  branch of the tree $\mc{A}$ has length at most $\log_2(m_{j_0}^2)$. Since also
  $w(g_a)<m_{j_0}^2$ for every $a\in\mc{A}$,
  our assumption entails that each non-maximal $a\in\mc{A}$
   has at most $m_{j_0}^2Q_{j_0}$ immediate successors. Thus
    $\#\supp g_1 \le (m_{j_0}^2Q_{j_0})^{\log_2(m_{j_0}^2)}$.
    Using the growth condition concerning the sequence $(n_j)_j$ (see Definition \ref{D13})
    we derive that
   \begin{equation} \label{sam2}
  |g_1(\frac{1}{n_{j_0}}\sum\limits_{k=1}^{n_{j_0}}e_k)|\le \frac{1}{n_{j_0}}\cdot \#\supp g_1
   \le \frac{1}{n_{j_0}}\cdot m_{j_0}^{4\log_2(m_{j_0})}Q_{j_0}^{2\log_2(m_{j_0})}\le\frac{1}{m_{j_0}^2}.
  \end{equation}
  From \eqref{sam1}, \eqref{sam2} we conclude that
  $|g(\frac{1}{n_{j_0}}\sum\limits_{k=1}^{n_{j_0}}e_k)|\le \frac{2}{m_{j_0}^2}$.
 \end{proof}

 \begin{proposition}[basic inequality]\label{P1}
 Let $(x_k)_{k=1}^{n_{j_0}}$  be a $(C,\e)$ RIS in $\eqs_D$ with associated sequence
  $(j_k)_{k=1}^{n_{j_0}}$.
   Then for every $f\in D$
    there exists a functional $g\in W$ such that the following
  conditions are fulfilled.
  \begin{enumerate}
  \item[(1)] If $f$ is  of type I then
  either $w(g)=w(f)$ or $g=e_r^*$ or $g=0$.
  \item[(2)] $|f(\frac{1}{n_{j_0}}\sum\limits_{k=1}^{n_{j_0}}x_k)|
     \le C
     \big(g(\frac{1}{n_{j_0}}\sum\limits_{k=1}^{n_{j_0}}e_k)+\e\big)$.
  \end{enumerate}

    If we additionally assume that $j_0<j_1$  and that for every subinterval $I$ of the interval
  $\{1,2,\ldots,n_{j_0}\}$  with $\# (I)\ge m_{j_0}^2Q_{j_0}$ and for every functional
  $h\in K^{j_0}\cdot D$
  it holds that $|h(\sum\limits_{k\in I}x_k)|\le C\cdot \e \cdot\# I$, then the functional
  $g\in W$ may  selected to satisfy in addition the following property:
   \begin{enumerate}
  \item[(3)] The functional $g$ admits a tree
  $(g_\beta)_{\beta\in\mc{B}}$ with the property that for every $\beta\in\mc{B}$
  with $g_\beta$ of type I and $w(g_\beta)<m_{j_0}^2$, the node $\beta$ has at most
  $m_{j_0}^2Q_{j_0}$ immediate successors.
  \end{enumerate}
 \end{proposition}
 \begin{proof}[\bf Proof.]
 We begin with the proof of the first part of the proposition (without the
 additional assumption).
 We fix a tree $(f_a)_{a\in \mc{A}}$ of the functional $f$ (Definition \ref{D4}).
 Using backward induction
  we shall define, for every $a\in\mc{A}$ and every
 subinterval $J$ of $\{1,2,\ldots,n_{j_0}\}$, a functional
 $g_a^J\in W$ such that the following conditions are satisfied.
  \begin{enumerate}
  \item[(i)] If $f_a$  of type I  then
  either $g_a^J=e_r^*$ or $g_a^J=0$ or $g_a^J$ is of type I and takes the form
   $g_a^J=\frac{1}{w(f_a)}(\sum\limits_{i=1}^d
 g_{\beta_i}^{J_i}+\sum\limits_{k\in J_0}e_k^*)$ where
   $S_a=\{\beta_1,\ldots,\beta_d\}$, each $J_i$ for $1\le i\le d$ is a subinterval of $J$, $J_0\subset J$
   and  the sets $J_0,J_1,\ldots, J_d$ are pairwise disjoint.
  \item[(ii)] $|f_a(\sum\limits_{k\in J}x_k)|
     \le C
     \big(g_a^J(\sum\limits_{k\in J}e_k)+\e\cdot\#J\big)$.
  \item[(iii)] The functional $g_a^J$ has nonnegative coordinates and $\supp g_a^J\subset J$.
    \end{enumerate}
  When the inductive construction is completed, the functional
  $g=g_0^{J_0}$, where $0\in A$ is the root of the tree $\mc{A}$ and
  $J_0=\{1,\ldots,n_{j_0}\}$, satisfies the conclusion of the
  proposition.

 The first inductive step concerns $a\in\mc{A}$ which are maximal.
 Then $f_a\in G$. We set $g_a^J=0$ for every subinterval $J$.
 Since  $\|x_k\|_G\le\e$ for each $k$
  (condition (i) of Definition \ref{D6}) it follows that
  $|f_a(\sum\limits_{k\in J}x_k)|\le \sum\limits_{k\in
 J}|f_a(x_k)|\le \e\cdot \#J$  and thus condition (ii) is
 satisfied.

 Let us pass to the general inductive step. We distinguish two
 cases.

 {\bf Case 1.}
 The functional $f_a$ is of type II,
 $f_a=\sum\limits_{i=1}^d\lambda_{\beta_i}f_{\beta_i}$.\\
 For every subinterval  $J$ we set
 $g_a^J=\sum\limits_{i=1}^d\lambda_{\beta_i}g_{\beta_i}^J$.
 Then, using the inductive hypothesis we get that
 \begin{eqnarray*}
  |f_a(\sum\limits_{k\in J}x_k)| & \le &
  \sum\limits_{i=1}^d\lambda_{\beta_i}|f_{\beta_i}(\sum\limits_{k\in
  J}x_k)|\le
  \sum\limits_{i=1}^d\lambda_{\beta_i}
  C\big(g_{\beta_i}^J(\sum\limits_{k\in J}e_k)+\e\cdot \#J\big)\\
   & = &  C\big(g_a^J(\sum\limits_{k\in J}e_k)+\e\cdot \#J\big).
 \end{eqnarray*}

 {\bf Case 2.}
 The functional $f_a$ is of type I,
 $f_a=\frac{1}{w(f_a)}(f_{\beta_1}+f_{\beta_2}+\cdots+f_{\beta_d})$
 (where
 $w(f_a)=m_{j^a_1}\cdot \ldots\cdot m_{j^a_r}$ and
 $d\le n_{j^a_1}+\cdots+n_{j^a_r}$ for some $j^a_1,\ldots,j^a_r\in\N$).\\
  Fix $J$  a subinterval of
 $\{1,\ldots,n_{j_0}\}$. We distinguish three subcases.

{\bf Subcase 2a.} $m_{j_{k_0}}\le w(f_a)<m_{j_{k_0+1}}$ for some $k_0\in J$.\\
 Then for $k\in J$ with $k<k_0$ we have that $m_{j_{k+1}}\le
 m_{j_{k_0}}\le w(f_a)$ and thus, using conditions (i), (ii) of Definition
 \ref{D6}, we get that
 \[  |f_a(x_k)|\le \frac{1}{w(f_a)}\|x_k\|_{\ell_1}
 \le\frac{1}{m_{j_{k+1}}}\cdot C\cdot \#\supp(x_k) \le C\cdot \e.\]
 For $k\in J$ with $k>k_0$, using conditions (ii), (iii) of Definition \ref{D6}, we get that
 \[ |f_a(x_k)|\le \frac{C}{w(f_a)}\le  \frac{C}{m_{j_1}}\le C\cdot
 \e.  \]
 If $\ran(f_a)\cap\ran(x_{k_0})\neq \emptyset$ we set  $g_a^J=e_{k_0}^*$,
  otherwise we set $g_a^J=0$.
  The inductive
 conditions are easily established.

 {\bf Subcase 2b.} $m_{j_{k+1}}\le w(f)$ for every $k\in J$.\\
 Then, as before, $|f_a(x_k)|\le C\cdot \e$ for every $k\in J$ and
 we set $g_a^J=0$.

  {\bf Subcase 2c.} $w(f_a)<m_{j_k}$ for all $k\in J$.\\
 Let $E_i=\ran f_{\beta_i}$ for $i=1,\ldots, d$. We set
 \[ J_0=\{k\in J: \; \ran(x_k) \mbox{ intersects at least two
 }E_i,\; i=1,\ldots,d\}\]
 and for $i=1,\ldots,d$ we set
\[ J_i=\{k\in J\setminus J_0:\; \ran(x_k) \mbox{ intersects
}E_i\}.\]
 We observe that $\#J_0\le d$ and  that each $J_i$ is a
 subinterval of $J$. From our inductive hypothesis the functionals
 $(g_{\beta_i}^{J_i})_{i=1}^d$ have been  defined satisfying
 conditions (i), (ii), (iii).  The family
 $\{J_i:\;i=1,\ldots,d\}\cup \{\{k\}:\;k\in J_0\}$ consists of
 pairwise disjoint intervals while its cardinality does not exceed
 $2d$.
  We define
 \[ g_a^J=\frac{1}{w(f_a)}(\sum\limits_{i=1}^d
 g_{\beta_i}^{J_i}+\sum\limits_{k\in J_0}e_k^*).\]
 From our definitions it follows that $g_a^J\in W$ (see Remark \ref{R4}).
 Observe, for later use, the following. If $w(f_a)<m_{j_0}^2$ and $f_a\not\in K^{j_0}\cdot D$
 then   Corollary \ref{C5}  entails that $d\le n_{j_0-1}\log_2(m_{j_0}^2)$ and hence
  $2d\le   4n_{j_0-1}\log_2(m_{j_0})\le m_{j_0}^2Q_{j_0}$.

 Since $w(f_a)<m_{j_k}$ for every $k\in J$, condition (iii) of
 Definition \ref{D6} implies that
 $|f_a(x_k)|\le\frac{C}{w(f_a)}$ for every $k\in J$. From this and
 from
 our inductive hypotheses we get that
 \begin{eqnarray*}
 |f_a(\sum\limits_{k\in J}x_k)| & \le & \sum\limits_{k\in
 J_0}|f_a(x_k)|+
 \frac{1}{w(f_a)}\sum\limits_{i=1}^d|f_{\beta_i}(\sum\limits_{k\in
 J_i}x_k)| \\
 &\le & \sum\limits_{k\in
 J_0}\frac{C}{w(f_a)}+\frac{1}{w(f_a)}\sum\limits_{i=1}^dC\big(g_{\beta_i}^{J_i}(\sum\limits_{k\in
 J_i}e_k)+\e\cdot \# J_i\big)  \\
 &\le & C\big(g_a(\sum\limits_{k\in J}e_k)+\e\cdot \#J\big).
 \end{eqnarray*}
 This completes the proof of the general inductive step and
 finishes the  proof of the first part of the proposition.

 Next, we proceed with the proof of the second part of the proposition,
 where the additional assumption is made. Then in addition to conditions (i), (ii), (iii)
 we require the following.

 \begin{enumerate}
  \item[(iv)]  If $g^J_a$ is of type $I$ with $w(g^J_a)<m_{j_0}^2$ then
   $\#\{i:\; g^{J_i}_{\beta_i}\neq 0\}+\#J_0\le m_{j_0}^2 Q_{j_0}$ (where
  the notation comes from condition (i)).
  \end{enumerate}

  The  procedure remains the same
   for the first inductive step
   and  for Case 1, Subcase 2a and Subcase 2b in the general inductive step.
  The difference concerns Subcase 2c
   (i.e. when $w(f_a)<m_{j_k}$ for all $k\in J$)
  where we distinguish two subsubcases.

  {\bf Subsubcase 2cA.}  $f_a\in K^{j_0}\cdot D$ and $\# J\ge m_{j_0}^2Q_{j_0}$.\\
   We set $g^J_a=0$ and our additional assumption yields
   condition (ii).

  {\bf Subsubcase 2cB.}  $f_a\not\in K^{j_0}\cdot D$ or $\# J< m_{j_0}^2Q_{j_0}$.\\
  We proceed exactly as in the proof of Subcase 2c in the first part of the proposition.

  We have to examine condition (iv).
  Observe that, if $g^J_a$ is of type I with $w(g^J_a)<m_{j_0}^2$ there are two cases. Either
  $f_a\not\in K^{j_0}\cdot D$ and $w(f_a)=w(g^J_a)<m_{j_0}^2$, in which
  case (see the proof of  Subcase 2c) the sum in the left side of (iv)
     does not exceed  $2d$ which is at most $m_{j_0}^2Q_{j_0}$, or
  $\# J< m_{j_0}^2Q_{j_0}$, in which case the same upper bound is derived from
  the fact that $\supp g^J_a\subset J$.

  This completes the proof of the general inductive step and the proof of
   the second part of the proposition.
 \end{proof}

 \begin{corollary}\label{C6}
 Let $(x_k)_{k=1}^{n_{j_0}}$ be a $(C,\e)$ RIS with
 $\e\le\frac{1}{m_{j_0}^2}$. Then for $f\in D$ of type I we have that
  \[   |f(\frac{1}{n_{j_0}}\sum\limits_{k=1}^{n_{j_0}}x_k)|\le
     \left\{ \begin{array} {l@{\quad} l}
           \frac{ 3C}{w(f)m_{j_0}} & \mbox{ if }w(f)<m_{j_0} \\[4mm]
   C(\frac{1}{w(f)}+\frac{1}{m_{j_0}^2}) & \mbox{ if }w(f)\ge
   m_{j_0}.
   \end{array}\right.     \]
   In particular $\|\frac{1}{n_{j_0}}\sum\limits_{k=1}^{n_{j_0}}x_k\|\le
   \frac{2C}{m_{j_0}}$.
 \end{corollary}
 \begin{proof}[\bf Proof.]
 From the basic inequality (Proposition \ref{P1}) there exists $g\in W$ with either
  $w(g)=w(f)$ or $g= e_r^*$ or $g=0$,
 such that
 $|f(\frac{1}{n_{j_0}}\sum\limits_{k=1}^{n_{j_0}}x_k)|
     \le C
     \big(g(\frac{1}{n_{j_0}}\sum\limits_{k=1}^{n_{j_0}}e_k)+\e\big)$.
     From Lemma \ref{L21} we get that if $w(f)<m_{j_0}$ then
 $|f(\frac{1}{n_{j_0}}\sum\limits_{k=1}^{n_{j_0}}x_k)|\le
 C(\frac{2}{w(f)m_{j_0}}+\frac{1}{m_{j_0}^2})\le\frac{3C}{w(f)m_{j_0}}$,
 while if $w(f)\ge m_{j_0}$ then
 $|f(\frac{1}{n_{j_0}}\sum\limits_{k=1}^{n_{j_0}}x_k)|\le
 C(\frac{1}{w(f)}+\frac{1}{m_{j_0}^2})$.
 \end{proof}

 \begin{definition} A vector $x\in \eqs_D$ is said to be a $C-\ell^k_1$
 average if $x$ takes the form
 $x=\frac{1}{k}\sum\limits_{i=1}^kx_i$,
 with $\|x_i\|\le C$ for each $i$, $x_1<\cdots<x_k$ and $\|x\|\ge 1$.
 \end{definition}

 \begin{lemma}\label{L1} Let $Y$ be a block subspace of $\eqs_D$ and $k\in \N$ .
 Then there exists a vector $x\in Y$ which is a $2-\ell_1^k$
 average.
 \end{lemma}
 For a proof we refer to \cite{ArTo} Lemma II.22.

 \begin{lemma}\label{L2} If $x$ is a $C-\ell_1^k$ average, $d\le k$ and $E_1<\cdots
 <E_d$ is any sequence of intervals, then
 $\sum\limits_{i=1}^d\|E_ix\|\le C(1+\frac{2d}{k})$.
 \end{lemma}
 For a proof we refer to \cite{ArTo} Lemma II.23.

 \begin{remark}\label{R5}
 (i) If $x$ is a $C-\ell_1^{n_j}$ average and $f$ is of
 type I  with
 $w(f)<m_j$, then from Lemma \ref{L2} and Corollary \ref{C2} we get that
  $|f(x)|\le \frac{2C}{w(f)}$.\\
  (ii) Suppose that \seq{x}{k} is a block sequence in $\eqs_D$, such
 that each $x_k$ is a $C-\ell_1^{n_{j_k}}$ average
  for an increasing sequence $(j_k)_{k\in\N}$ and
 $\|x_k\|_G\le\e$ for all $k$.
 From (i), if $f\in D$ with $w(f)<m_{j_k}$ then
 $|f(x_k)|\le \frac{2C}{w(f)}$.
 Thus, we may easily select a subsequence of \seq{x}{k}
  which is a $(2C,\e)$ RIS.
  \end{remark}

 \begin{proposition}\label{P4}
 The space $\eqs_D$ is a strictly singular extension of $Y_G$ (i.e. the
 identity operator $I:\eqs_D\to Y_G$, where $Y_G$ is the completion of
 $(c_{00}(\N),\|\cdot\|_G)$, is strictly singular).
 \end{proposition}
 \begin{proof}[\bf Proof.] Assume the contrary.  Then we can find
 a block subspace $Z$ of $\eqs_D$ and $\e>0$ such that $\|z\|_G\ge
 \e\|z\|$ for every $z\in Z$. Pick $(z_n)_{n\in\N}$, $(z^*_n)_{n\in\N}$
 two  normalized block sequences in $Z$, $\eqs_D^*$ respectively, such that
 $\ran z_n=\ran z_n^*$ and $z_n^*(z_n)=1$. Passing to
 subsequences we may assume that the sequence $(z_n)_{n\in\N}$ is
 weakly Cauchy in $Y_G$, hence the sequence of its successive
 differences, i.e. the sequence \seq{x}{n} defined by the rule
 $x_n=z_{2n-1}-z_{2n}$ is weakly null in $Y_G$. Thus
 $\lim\limits_{n}\chi_{\N}(x_n)=0$.
  Passing to further
 subsequences we may  assume that
 $\sum\limits_{n=1}^{\infty}|\chi_{\N}(x_n)|<\frac{1}{2}$.
 Then $\|\sum\limits_{i=1}^kx_i\|_G<\frac{9}{2}$ for every
 $k\in\N$ (see the proof of Lemma \ref{L12}).

 Let now $j\in \N$. We set
 $x=\frac{1}{n_{2j}}\sum\limits_{i=1}^{n_{2j}}x_i$ and
 $x^*=\frac{1}{m_{2j}}\sum\limits_{i=1}^{n_{2j}}z^*_{2i-1}$.
 Then $\|x^*\|\le 1$ (see Remark \ref{R3}), hence
 \[\frac{1}{n_{2j}}\cdot\frac{9}{2}\ge \|x\|_G\ge \e\cdot\|x\|
 \ge \e \cdot x^*(x)= \e\cdot \frac{1}{m_{2j}}.\]
 Since this happens for every $j$, for large $j$ we derive a
 contradiction.
  \end{proof}

   \begin{definition}\label{D7}
 Let $C\ge 1$ and $j\in\N$. A pair $(x,x^*)$ is said to be
 a $(C,2j)$ exact pair, provided that
 \begin{enumerate}
 \item[(i)] $x^*\in K$    with $w(x^*)=m_{2j}$ (i.e. $x^*\in K^{2j}$).
 \item[(ii)]  $\|x\|_G\le\frac{1}{m_{2j}^2}$, and for every $f\in D$ of type I,
 if $w(f)<m_{2j}$ then $|f(x)|\le
 \frac{3C}{w(f)}$, while if $w(f)\ge m_{2j}$ then $|f(x)|\le
 C(\frac{m_{2j}}{w(f)}+\frac{1}{m_{2j}})$ (in particular $\|x\|\le 2C$).
 \item[(iii)] $x^*(x)=1$ and $\ran x^*= \ran x$.
 \end{enumerate}
 \end{definition}

  \begin{lemma}\label{L18}
  For every  block subspace $Z$ of $\eqs_D$ and every $j\in\N$
   there exists a
  $(4,2j)$ exact pair $(z,z^*)$ with $z\in Z$.
  \end{lemma}
  \begin{proof}[\bf Proof.] We set
  $\e=\frac{1}{m_{2j}^3}$.
  Using Proposition \ref{P4}, we may select a block subspace $Z'$ of
  $Z$ such  that the restriction of the identity operator $I:\eqs_D\to Y_G$
  on $Z'$ has norm less than $\frac{\e}{2}$. From Lemma \ref{L1}
  we may select a block sequence $\seq{z}{k}$ in $Z'$ with each
  $z_k$ being a $2-\ell_1^{n_{j_k}}$  average for an
  increasing sequence $(j_k)_{k\in\N}$. Then $\|z_k\|_G\le \e$ for all $k$.
   From Remark \ref{R5} we may assume,
  passing to a subsequence, that $(z_k)_{k=1}^{n_{2j}}$ is a
  $(4,\e)$ RIS.
  For each $k$ we select $z_k^*\in D$ with $\ran(z_k^*)=\ran(z_k)$
  such that $z_k^*(z_k)\ge 1$.

  We set \[x=\frac{1}{n_{2j}}\sum\limits_{k=1}^{n_{2j}}z_k \quad\mbox{
  and }\quad z^*=\frac{1}{m_{2j}}\sum\limits_{k=1}^{n_{2j}}z^*_k
  .\]
  From Corollary \ref{C6} we get that
  $\frac{1}{m_{2j}}\le z^*(x)\le \|x\|\le \frac{8}{m_{2j}}$, thus we
  may select $\frac{1}{8}\le \theta \le 1$ such that
   $z^*(\theta m_{2j} x)=1$. We set $z=\theta m_{2j} x$.
   From the fact that
   $\|z\|_G \le  m_{2j}\cdot\e=\frac{1}{m_{2j}^2}$, and using Corollary \ref{C6}, we
    easily establish condition (ii) of Definition \ref{D7}.
    Hence $(z,z^*)$ is a $(4,2j)$ exact pair.
  \end{proof}

\begin{remark}\label{R8}
Exact pairs are one of the fundamental ingredients to show that
the space $\eqs_D$ is HI. This property will be proved in the next
section. A rather direct consequence of the previous lemma is that
$\ell_1(\N)$ is not embedded in the space $\eqs_D$. To see this,
observe that Lemma \ref{L18} yields that for every $j\in\N$  there
exists a normalized finite block sequence $(w_k)_{k=1}^{n_{2j}}$
in $Z$ such that $\|\frac{w_1+w_2+\cdots+w_{n_{2j}}}{n_{2j}}\|\le
\frac{4}{m_{2j}}$.
\end{remark}

   \begin{proposition}\label{P5}
    The space $\eqs_D^*$ is the closed lineal span of the pointwise
    closure of $G$, i.e.
    $\eqs_D^*=\overline{\spann}(\overline{G}^{w^*})
        =\overline{\spann}(\{e_n^*:\; n\in\N\}\cup \{\chi_\N\})$.
     Thus the space $\eqs_D$ is quasireflexive, with its codimension in the second
     dual being equal to one.
    \end{proposition}
   \begin{proof}[\bf Proof.]
  Suppose that the space $Z=\overline{\spann}(\overline{G}^{w^*})$
  is a proper subspace of $\eqs_D^*$. Then we may select an $x^*\in
  \eqs_D^*\setminus Z$ with $\|x^*\|=1$ and an $x^{**}\in \eqs_D^{**}$ such
  that $\|x^{**}\|=2$, $x^{**}(x^*)=2$ and $ Z\subset\Ker x^{**}$.
 Since the space $\eqs_D$ contains no isomorphic copy of $\ell_1(\N)$
 (Remark \ref{R8}), from a theorem of Odell and Rosenthal  (\cite{OR}), we
 can choose $(x_k)_{k\in\N}$ with $\|x_k\|\le 2$ such that
 $x_k\stackrel{w^*}{\longrightarrow} x^{**}$. We may assume that
 \seq{x}{k} is a block sequence (since $e_n^*\in Z$ for all $n$)
 and that $x^*(x_k)>1$ for each $k$ (since $x^*(x_k)\to x^{**}(x^*)=2$).

 From the fact that $\Ext (B_{Y_G^*})\subset
 \overline{G}^{w^*}\subset Z$ we get that the sequence \seq{x}{k}
 is weakly null in $Y_G$, thus we may select a block sequence  \seq{y}{k}  of
 \seq{x}{k} such that $\|y_k\|_G\to 0$ and such that each $y_k$ is
 a convex combination of \seq{x}{k}. Passing to a subsequence
 of \seq{y}{k}, we
 may assume that $\|y_k\|_G<\e$ for all $k$ where
 $\e=\frac{1}{m_{2j}^3}$ for some fixed $j$.
 Notice also that $x^*(y_k)>1$ for all $k$. We may construct
 $(z_k)_{k\in\N}$ a  sequence of $2-\ell_1^{n_{j_k}}$ averages of increasing length
 with each $z_k$ being an average of \seq{y}{k}.
 Thus $\|z_k\|_G<\e$ for all $k$ and passing to a subsequence we
 may assume that $(z_k)_{k=1}^{n_{2j}}$ is a $(4,2j)$ RIS.
 From Corollary \ref{C6} we get that $\|z\|\le \frac{8}{m_{2j}}<1$,
 where $z=\frac{1}{n_{2j}}\sum\limits_{k=1}^{n_{2j}}z_k$. On the
 other hand the action of the functional $x^*$ entails that
  $\|z\|\ge x^*(z)>1$
 (since $z$ is a convex combination of \seq{y}{k}), a contradiction.
  \end{proof}

\section{Dependent sequences and the HI property of $\eqs_{D}$}

 In this section we define the dependent sequences which will be
 used for the proof that the spaces $\eqs_D$, $\eqs_D^*$ are HI.
 The main goal is to prove Lemma \ref{L15} and then to use the
 basic inequality to obtain upper estimates of the norm for
 certain averages of block sequences (Proposition \ref{P3}).
 As a consequence, we obtain that the space $\eqs_D$ is Hereditarily
 Indecomposable.

 \begin{definition}\label{D12} Let $C\ge 1$ and $j\ge 2$.
 A double sequence $\chi=(x_k,x_k^*)_{k=1}^{n_{2j-1}}$ is said to be a $(C,2j-1)$
 dependent sequence provided there exists a sequence
 $(j_k)_{k=1}^{n_{2j-1}}$ such that
 \begin{enumerate}
 \item[(i)] The sequence $(x_k^*)_{k=1}^{n_{2j-1}}$ is
 a $2j-1$ special sequence (Definition \ref{D8})
 with $w(x_k^*)=m_{2j_k}$ for each $k$.
 \item[(ii)] Each $(x_k,x_k^*)$ is a $(C,2j_k)$ exact pair.
 \end{enumerate}
 For $\chi$ as above we shall denote by $\phi_\chi$ the functional
 $\phi_\chi=\frac{1}{m_{2j-1}}\sum\limits_{k=1}^{n_{2j-1}}x_k^*$.
 Observe that $\phi_\chi\in K^{2j-1}\subset D$ (Remark \ref{R6}(iii)), hence $\|\phi_{\chi}
 \|\le 1$.
 \end{definition}

 \begin{remark}\label{R24}
  From Definitions \ref{D8}, \ref{D7} and \ref{D12} and the growth condition
  concerning the coding function $\sigma$, it follows that when
  $(x_k,x_k^*)_{k=1}^{n_{2j-1}}$ is a $(C,2j-1)$ dependent sequence,
  the sequence
  $(x_k)_{k=1}^{n_{2j-1}}$ is a $(3C,\frac{1}{m_{2j-1}^2})$ RIS.
 \end{remark}

 \begin{lemma}\label{L13}
 Let $f\in D$ and  let $\chi=(x_k,x_k^*)_{k=1}^{n_{2j-1}}$ be a
 $(C,2j-1)$  dependent sequence.
 Then for any subinterval $I$  of the interval $\{1,2,\ldots,n_{2j-1}\}$,
 we have that
 \[ |(f\cdot \phi_{\chi})\big(\sum\limits_{k\in
 I}(-1)^{k+1}x_k\big)|\le \frac{4C}{m_{2j-1}^2}\cdot\#(I)+2C\cdot Q_{2j-1}.  \]
 \end{lemma}
 \begin{proof}[\bf Proof.] We fix a
 tree $(f_a)_{a\in\mc{A}}$ of the functional $f$ (Definition
 \ref{D4}). For each $a\in\mc{A}$ we set
 $w_a=\prod\{w(f_{\beta}): \;\beta\in \mc{A},\; \beta\prec
 a,\;f_{\beta} \mbox{ is of type I }\}$. We consider the following
 subsets of $\mc{A}$.
 \begin{eqnarray*}
  \mc{D}_0&= &\{a\in\mc{A}:\; f_a \mbox{ is of type 0 and }w_a<m_{2j-1}\;\}\\
  \mc{D}_1& =&\{a\in\mc{A}:\; f_a \mbox{ is of type I, }w_a<m_{2j-1} \mbox{ and }w(f_a)\ge m_{2j-1}\;
  \}\\
  \mc{D}_2&=&\{a\in\mc{A}:\; f_a \mbox{ is of type 0 or of type I, }w_a\ge m_{2j-1} \mbox{ and }\\
       & &\mbox{ for every }\beta\prec a \mbox{ with $f_\beta$ of type I, }
              w_\beta<m_{2j-1} \mbox{ and } w(f_{\beta})\le
              m_{2j-1}\}.
  \end{eqnarray*}
  We set $\mc{D}=\mc{D}_0\cup\mc{D}_1\cup\mc{D}_2$ and
  $\mc{B}=\{\beta\in\mc{A}:\mbox{ there exists }\delta\in\mc{D} \mbox{ with }
  \beta \preceq \delta\}$.
  The following properties hold.
  \begin{enumerate}
  \item[(i)] The sets $\mc{D}_0$, $\mc{D}_1$, $\mc{D}_2$ are
  pairwise disjoint.
  \item[(ii)] The nodes of  $\mc{D}$ are pairwise incomparable
  with respect to the order of $\mc{A}$. Moreover, the set $\mc{D}$
  is a  maximal subset of $\mc{A}$ with this property.
  In order to see the later it is enough to show that every branch
  of the tree $\mc{A}$ contains a member of $\mc{D}$. Indeed, fix
  a branch of $\mc{A}$ and let $a_0\prec a_1\prec \cdots\prec a_k$
  be the members of this branch for which the corresponding
  functionals are of type I or of type 0. If the set
  $\big\{i\in\{0,1,\ldots,k\}:\; w_{a_i} \ge m_{2j-1}\mbox{ or
  }w(f_{a_i})\ge m_{2j-1}\big\}$ is nonempty and $i_0$ is its
  minimum then $a_{i_0}\in\mc{D}$, while if this set is
  empty then $a_k\in\mc{D}$.\\
  Thus
  $\mc{B}$ is a complete subtree of $\mc{A}$, while a node
  $\beta\in\mc{B}$ is maximal of $\mc{B}$ if and only if
  $\beta\in\mc{D}$.
  \item[(iii)] For every $\delta\in\mc{D}$ the set
  $\{\beta\in\mc{B}:\;\beta \prec \delta \mbox{ and $f_{\beta}$  is of type I}
  \}$ has cardinality at most $\log_2(m_{2j-1})$.
 Indeed, let $\{\beta_0\prec \beta_1\prec\cdots\prec \beta_{d-1}\}$ be the later set.
  Then $w_{\beta_{d-1}}<m_{2j-1}$ and since
  $w_{\beta_{d-1}}=w(f_{\beta_0})\cdot\ldots\cdot
  w(f_{\beta_{d-2}})\ge 2^{d-1}$, it follows that
  $d-1<\log_2(m_{2j-1})$, therefore $d\le \log_2(m_{2j-1})$.
  \item[(iv)] For every $\beta\in\mc{B}$ with
  $\beta\not\in\mc{D}$, such that  the functional $f_{\beta}$ is of type
  I, it holds that $w(f_{\beta})<m_{2j-1}$. It follows from Corollary
  \ref{C2} that the node $\beta\in\mc{B}$ has at most
  $n_{2j-2}\log_2(m_{2j-1})$ immediate successors.
  \end{enumerate}

  Using arguments similar to those in the proof of Lemma
  \ref{L22},
  we may prove the following. The functional $f$ takes the form
  $f=\sum\limits_{i=1}^d\lambda_if_i$ as a convex combination, such
  that for each $i$ there exists a family of functionals
   $(f^i_{\beta})_{\beta\in\mc{B}^i}$, indexed by a finite tree
   $\mc{B}^i$ and an order preserving map
   $\Phi^i:\mc{B}^i\to\mc{B}$ with the following properties.

   \begin{enumerate}
   \item[(v)] The tree $\mc{B}^i$ has a unique root
  $0\in\mc{B}^i$  and  $f_0^i=f_i$.
  \item[(vi)] For every maximal node $\beta\in\mc{B}^i$, $\Phi^i(\beta)$ is
  a maximal node of $\mc{B}$ (and thus $\Phi^i(\beta)\in\mc{D}$) and
 $f^i_\beta=f_{\Phi^i(\beta)}$. In particular, $f^i_\beta\in D$.
 \item[(vii)] For every   $a\in\mc{B}^i$ which is non-maximal, the
 functionals  $f^i_a$ and $f_{\Phi^i(a)}$ are of type I with
  $w(f^i_a)=w(f_{\Phi^i(a)})$ and  $\# S^i_a=\# S_{\Phi^i(a)}$, where
  $S^i_a$ is the set of immediate
 successors of  $a\in\mc{B}^i$ and  $S_\gamma$ denotes the set of immediate
 successors for a $\gamma\in \mc{B}$.
  Thus
  $f^i_a=\frac{1}{w(f^i_a)}(f^i_{\beta_1}+\cdots+f^i_{\beta_d})$
  with $f^i_{\beta_1}<\cdots<f^i_{\beta_d}$ and
  $d<n_{2j-2}\log_2(m_{2j-1})$ (see property (iv) above).
   Moreover, defining
  $w_\beta^i=\prod\{w(f^i_a):\;a\in\mc{B}^i, \;a\prec \beta\}$ for
  each $\beta\in\mc{B}^i$, we have that $w^i_\beta=w_{\Phi^i(\beta)}$.
 \item[(viii)] Denoting by $\mc{B}^i_{\max}$ the set of maximal nodes of
 the tree $\mc{B}^i$,  the functionals $(f^i_{\beta})_{\beta\in\mc{B}^i_{\max}}$
 are successive and
 $f=\sum\limits_{\beta\in\mc{B}^i_{\max}}\frac{1}{w^i_{\beta}}f^i_{\beta}$.
 \end{enumerate}
 We notice that for $a\in\mc{B}^i$ which is non-maximal, the
 functional $f_\beta^i$ may not belong to the norming set $D$ of the space $\eqs_D$,
  but  certainly belongs to the unconditional frame of the space,
  i.e. to the minimal subset of $c_{00}(\N)$ which contains $\{\pm
  e_n^*:\;n\in\N\}$, is closed under the $(\mc{A}_{n_j},\frac{1}{m_j})$ operations
  and is rationally convex.

  From properties (iii), (iv), (vii) and the fact that the map
  $\Phi^i$ is order preserving we get the following.
  \begin{enumerate}
  \item[(ix)] The cardinality of the   set $\mc{B}^i_{\max}$
   of maximal nodes of  $\mc{B}^i$ does not exceed the number
    $Q_{2j-1}=(n_{2j-2}\log_2(m_{2j-1}))^{\log_2(m_{2j-1})}$.
  \end{enumerate}

  Since the functional $f$ equals to the convex combination
  $\sum\limits_{i=1}^d\lambda_if_i$, in order to prove the lemma,
  it is enough to show that for each $i\in\{1,\ldots,d\}$, it holds that
  \begin{equation}\label{ineq1} |(f_i\cdot \phi_{\chi})\big(\sum\limits_{k\in
 I}(-1)^{k+1}x_k\big)|\le \frac{4C}{m_{2j-1}^2}\cdot\#(I)+2C\cdot Q_{2j-1}.
 \end{equation}
  We fix $i\in\{1,\ldots,d\}$ and  we set $E_\beta=\ran(f^i_\beta)$ for each
 $\beta\in\mc{B}^i_{\max}$. We partition the interval
 $I$ as follows
 \[I_1=\{k\in I:\; E_{\beta}\cap\ran(x_k)\neq \emptyset\text{ for
 at most one }\beta\in\mc{B}^i_{\max}\}\]
  \[I_2=\{k\in I:\; E_{\beta}\cap\ran(x_k)\neq \emptyset\text{ for
 at least two }\beta\in\mc{B}^i_{\max}\}.\]
  For each $\beta\in\mc{B}^i_{\max}$ we set
 $I_\beta=\{k\in I_1:\; E_{\beta}\cap\ran(x_k)\neq
 \emptyset\}$. We observe the following.
  \begin{enumerate}
  \item[(x)]
   Each $I_\beta$ is an interval and the intervals
   $(I_{\beta})_{\beta\in \mc{B}^i_{\max}}$ are pairwise disjoint.
 \item[(xi)]  For each $\beta\in\mc{B}^i_{\max}$ we have that
 $E_{\beta}\cap\ran(x_k)\neq \emptyset$ for at most two $k\in
 I_2$.
  \end{enumerate}
  For $p=0,1,2$ we set
  $\mc{B}^{i,p}_{\max}=\{a\in\mc{B}^i_{\max}:\;\Phi^i(a)\in\mc{D}_p\}$.
  Observe that the sets
  $\mc{B}^{i,0}_{\max},\mc{B}^{i,1}_{\max}, \mc{B}^{i,2}_{\max}$ form a partition of
  $\mc{B}^i_{\max}$.
  The proof of \eqref{ineq1} (and the proof of the whole lemma) will
  be complete after showing the following.

 \begin{eqnarray}
 \label{ineq2}|(f_i\cdot \phi_{\chi})(\sum\limits_{k\in I_2}(-1)^{k+1}x_k)|& \le & C\cdot
 Q_{2j-1}\\
 \label{ineq3}  |\big((\sum\limits_{\beta\in\mc{B}^{i,2}_{\max}}\frac{1}{w^i_{\beta}}f^i_{\beta})
 \cdot \phi_{\chi}\big)(\sum\limits_{k\in I_1}(-1)^{k+1}x_k)| &\le
 &  \frac{2C}{m_{2j-1}^2}\cdot\#(I)\\
  \label{ineq4} |\big((\sum\limits_{\beta\in\mc{B}^{i,1}_{\max}}\frac{1}{w^i_{\beta}}f^i_{\beta})
 \cdot \phi_{\chi}\big)(\sum\limits_{k\in I_1}(-1)^{k+1}x_k)|&\le&
 \frac{2C}{m_{2j-1}^2}\cdot\#(I)\\
  \label{ineq5} |\big((\sum\limits_{\beta\in\mc{B}^{i,0}_{\max}}\frac{1}{w^i_{\beta}}f^i_{\beta})
 \cdot \phi_{\chi}\big)(\sum\limits_{k\in I_1}(-1)^{k+1}x_k)|&\le&
 C\cdot Q_{2j-1}.
 \end{eqnarray}

 Let us first prove \eqref{ineq2}.
  From properties (viii), (ix), (xi),    the fact that
 $\|x_k\|\le 2C$ and since $f^i_{\beta}\cdot x_k^* \in D$ for each
  $\beta\in \mc{B}^i_{\max}$ and each $k$, we get
 that
 \begin{eqnarray*}
 |(f_i\cdot \phi_{\chi})(\sum\limits_{k\in I_2}(-1)^{k+1}x_k)| & \le &
   \sum\limits_{\beta\in\mc{B}^i_{\max}}\frac{1}{w^i_{\beta}}|(f^i_{\beta}
   \cdot \phi_{\chi})(\sum\limits_{k\in I_2}(-1)^{k+1}x_k)|\\
   & \le & \frac{1}{m_{2j-1}}\sum\limits_{\beta\in\mc{B}^i_{\max}}
   \sum\limits_{k\in I_2}|( f^i_{\beta}\cdot x_k^*)(x_k)|\\
   &\le& \frac{1}{m_{2j-1}}
   \cdot\#(\mc{B}^i_{\max})\cdot 2\cdot\max\limits_k\|x_k\|\\
   &\le &   \frac{4C\cdot Q_{2j-1}}{m_{2j-1}}\le C\cdot Q_{2j-1}.
 \end{eqnarray*}

 We pass to the proof of \eqref{ineq3}.
  From (vii) and from the definitions of $\mc{B}_{\max}^{i,2}$, $\mc{D}_2$ we get
 that for $k\in\mc{B}_{\max}^{i,2}$ it holds that
 $w^i_{\beta}=w_{\Phi^i(\beta)}\ge m_{2j-1}$.
 From this and from the fact that the sets $(I_{\beta})_{\beta\in\mc{B}^i_{\max}}$ are
 pairwise disjoint subsets of $I$, we get that
 \begin{eqnarray*}
 |\big((\sum\limits_{\beta\in\mc{B}^{i,2}_{\max}}\frac{1}{w^i_{\beta}}f^i_{\beta})
 \cdot \phi_{\chi}\big)(\sum\limits_{k\in I_1}(-1)^{k+1}x_k)| &
 \le &
 \sum\limits_{\beta\in\mc{B}^{i,2}_{\max}} \frac{1}{w^i_{\beta}}\frac{1}{m_{2j-1}}
 \sum\limits_{k\in I_1}
 | (f^i_{\beta}\cdot x_k^*)(x_k)|  \\
 & \le & \frac{1}{m_{2j-1}^2}\sum\limits_{\beta\in\mc{B}^{i,2}_{\max}}\sum\limits_{k\in
 I_{\beta}}\|x_k\| \\
 & \le &
 \frac{1}{m_{2j-1}^2}\cdot\#(\bigcup\limits_{\beta\in\mc{B}^{i,2}_{\max}}I_{\beta})\cdot2C\\
 & &
 \le \frac{2C}{m_{2j-1}^2}\cdot\#(I).
 \end{eqnarray*}

 Next we show \eqref{ineq4}.
 From the fact that each $(x_k,x_k^*)$ is
 a $(C,2j_k)$ exact pair,
 it follows that
  for every $\beta\in\mc{B}^{i,1}_{\max}$ we have that
 $|(f^i_{\beta}\cdot x_k^*)(x_k)|\le
 C(\frac{ m_{2j_k}}{w(f^i_\beta) w(x_k^*)}+\frac{1}{m_{2j_k}})
  \le C (\frac{ m_{2j_k}}{m_{2j-1} m_{2j_k}}+\frac{1}{m_{2j_k}})
 \le\frac{2C}{m_{2j-1}}$
 (we have used that  $w(f^i_\beta)=w(f_{\Phi^i({\beta})})\ge m_{2j-1}$
  as follows from (vii) and from the  definitions of $\mc{B}_{\max}^{i,1}$, $\mc{D}_1$).
  This implies that
 \begin{eqnarray*}
 |\big((\sum\limits_{\beta\in\mc{B}^{i,1}_{\max}}\frac{1}{w^i_{\beta}}f^i_{\beta})
 \cdot \phi_{\chi}\big)(\sum\limits_{k\in I_1}(-1)^{k+1}x_k)| &
 \le &
 \sum\limits_{\beta\in\mc{B}^{i,1}_{\max}}
 \sum\limits_{k\in I_{\beta}}\frac{1}{m_{2j-1}}
 | (f^i_{\beta}\cdot x_k^*)(x_k)|   \\
 & \le & \frac{1}{m_{2j-1}}\sum\limits_{\beta\in\mc{B}^{i,1}_{\max}}\sum\limits_{k\in
 I_{\beta}} \frac{2C}{m_{2j-1}}\\
 & \le &
  \frac{2C}{m_{2j-1}^2}\cdot\#(I).
 \end{eqnarray*}

 Finally, we shall show \eqref{ineq5}.
 For each $\beta\in\mc{B}^{i,0}_{\max}$
  the functional $f^i_{\beta}$ is of the form
 $f^i_{\beta}=\pm \chi_{J_{\beta}}$  for some interval $J_{\beta}$.
 Taking into account that $I_{\beta}$ is an interval and that $x_k^*(x_k)=1$ for each $k$,
 setting $k_1=\min I_\beta$, $k_2=\max I_{\beta}$ and $I_{\beta}'=I_{\beta}\setminus\{k_1,k_2\}$,
  we get that
 \begin{eqnarray*}
  |( f^i_{\beta}\cdot\phi_{\chi})(\sum\limits_{k\in
 I_\beta}(-1)^{k+1}x_k)| &\le & \frac{1}{m_{2j-1}}\big(\|x_{k_1}\|+\big|(\sum\limits_{k\in
 I_{\beta}'}x_k^*)
  (\sum\limits_{k\in
 I_{\beta}'}(-1)^{k+1}x_k)\big|\\
 & & \qquad\qquad\qquad\qquad\qquad\qquad  +\|x_{k_2}\|\big)\\
   & \le & \frac{1}{m_{2j-1}}\cdot(2C+1+2C)=\frac{4C+1}{m_{2j-1}}.
  \end{eqnarray*}
   Hence,
 \begin{eqnarray*}
 |\big((\sum\limits_{\beta\in\mc{B}^{i,0}_{\max}}\frac{1}{w^i_{\beta}}f^i_{\beta})
 \cdot \phi_{\chi}\big)(\sum\limits_{k\in I_1}(-1)^{k+1}x_k)| &
 \le &
 \sum\limits_{\beta\in\mc{B}^{i,0}_{\max}}\frac{1}{w^i_{\beta}}
 |( f^i_{\beta}\cdot\phi_{\chi})(\sum\limits_{k\in
 I_\beta}(-1)^{k+1}x_k)|\\
  &\le & \frac{4C+1}{m_{2j-1}}\cdot\#(\mc{B}^{i,0}_{\max})\le C\cdot Q_{2j-1}.
 \end{eqnarray*}

 From \eqref{ineq2}, \eqref{ineq3}, \eqref{ineq4}, \eqref{ineq5}
 we get \eqref{ineq1} and this completes the proof of the lemma.
 \end{proof}

  With the next lemma we pass from the action of   products of the form $f\cdot
 \phi_{\chi}$ for $f\in D$ on the vector $\sum\limits_{k\in
 I}(-1)^{k+1}x_k$,
  to the  action of $f\cdot \phi$ for $f\in D$ and arbitrary $\phi\in K$
  with
 $w(\phi)=m_{2j-1}$ on the same vector.

 \begin{lemma}\label{L15}
 Let  $f\in D$, let $\phi\in K^{2j-1}$  (which means that $\phi$ is
 of type I with $w(\phi)=m_{2j-1}$) and let $\chi=(x_k,x_k^*)_{k=1}^{n_{2j-1}}$
   be a $(C,2j-1)$ dependent sequence. Then for
  every  subinterval $I$  of the interval $\{1,2,\ldots,n_{2j-1}\}$, we have
 that \[ |(f\cdot \phi)\big(\sum\limits_{k\in
 I}(-1)^{k+1}x_k\big)|\le
                           \frac{5C}{m_{2j-1}^2}\cdot\#(I)+3C\cdot Q_{2j-1}.  \]
 \end{lemma}
 \begin{proof}[\bf Proof.]
  We may assume, without loss of generality, that the functional
 $f$ is either of type I or of type 0 (since every
  member of $D$ is s convex combination of such functionals).
  The functional $\phi$ takes the form
 \[\phi=\frac{1}{m_{2j-1}}(Ex_t^*+x_{t+1}^*+\cdots+x_{r-1}^*+f_r+f_{r+1}+\cdots+f_d)\]
 where
 $(x_1^*,x_2^*,\ldots,x_{r-1}^*,f_r,f_{r+1},\cdots,f_{n_{2j-1}})$ is
 some $2j-1$ special sequence,
   $w(f_r)=w(x_r^*)$, $f_r\neq x_r^*$, $d\le n_{2j-1}$ and $E$ is
   an interval of the form $[m,\max\supp x_t^*]$.
   We set
   \[\phi_1=\frac{1}{m_{2j-1}}(Ex_t^*+x_{t+1}^*+\cdots+x_{r-1}^*)\quad\mbox{and}
   \quad
   \phi_2=\frac{1}{m_{2j-1}}(f_r+f_{r+1}+\cdots+f_d).\]
   We observe that $\phi_1\cdot
   f=\big(\frac{1}{m_{2j-1}}(x_1^*+x_2^*+\cdots+x_{n_{2j-1}}^*)\big)\cdot \chi_{[\min E,\max\supp
   x_{r-1}^*]}\cdot f
   =\phi_{\chi}\cdot h$ where $h=\chi_{[\min E,\max\supp
   x_{r-1}^*]}\cdot f\in D$.
  From Lemma \ref{L13} it follows that
 \begin{equation}\label{eq5}
   |(f\cdot \phi_1)\big(\sum\limits_{k\in
 I}(-1)^{k+1}x_k\big)|\le \frac{4C}{m_{2j-1}^2}\cdot\#(I)+2C\cdot Q_{2j-1}.
 \end{equation}

 We select  $p$ such that $w(x_{p-1}^*)< w(f)\le w(x_p^*)$ (the
 adaptations in the rest of the  proof are obvious if no such a $p$ exists).
 From the injectivity of the coding function $\sigma$ and the
 definition of  special functionals (Definition \ref{D8})
 we get that
 the sets $\{w(f_{r+1}),\ldots,w(f_d)\}$ and
 $\{w(x_k^*):\;k=1,\ldots,n_{2j-1}\}$ are disjoint and both are
 subsets of the set $\{m_{2i}:\;i\in\N\}$.

 Let $k\in I$, $k< p-1$. Then for every $i\in\{r,\ldots,d\}$ we
 have that $w(f\cdot f_i)\ge w(f)>  w(x_{p-1}^*)\ge m_{2j_k}^5$,
 hence, using that $(x_k,x_k^*)$ is a
 $(C,2j_k)$ exact pair, we get that
  $|(f\cdot f_i)(x_k)|\le C(\frac{m_{2j_k}}{w(f\cdot f_i)}+\frac{1}{m_{2j_k}})\le
  \frac{2C}{m_{2j_k}}$. Thus
 \[|(f\cdot \phi_2)(x_k)|\le
 \frac{1}{m_{2j-1}}\sum\limits_{i=r}^d|(f\cdot f_i)(x_k)|
 \le \frac{1}{m_{2j-1}}\cdot n_{2j-1}\cdot \frac{2C}{m_{2j_k}}\le
 \frac{C}{m_{2j-1}^2}.\]

 Let now $k\in I$, $k>p$, $k\neq r$. We observe that for $i$ such that
 $w(f_i)<m_{2j_k}$ we also have that $w(f\cdot f_i)=w(f)\cdot
 w(f_i)\le  m_{2j_p}\cdot m_{2j_k-1}<m_{2j_k}$.
  Thus setting $J^k_-=\{i:\;r\le i\le d,\; w(f_i)<m_{2j_k}\}$,
  $J^k_+=\{i:\;r\le i\le d,\; w(f_i)> m_{2j_k}\}$ and taking into account that
  $(x_k,x_k^*)$ is a
 $(C,2j_k)$ exact pair, we get that
  \begin{eqnarray*}
  |(f\cdot \phi_2)(x_k)|  &\le&
  \frac{1}{m_{2j-1}}\big(\sum\limits_{i\in J^k_-}|(f\cdot f_i)(x_k)|+
             \sum\limits_{i\in J^k_+}|(f\cdot f_i)(x_k)|\big)\\
  &\le & \frac{1}{m_{2j-1}}\big(\sum\limits_{i\in J^k_-}\frac{3C}{w(f)\cdot w(f_i)}
              +
   \sum\limits_{i\in J^k_+}C(\frac{m_{2j_k}}{w(f)\cdot w(f_i)}+\frac{1}{m_{2j_k}})\big)\\
 &  \le &
  \frac{C}{m_{2j-1}}\big(\sum\limits_{i\in J^k_-}\frac{3}{w(f_i)}+m_{2j_k}
       \sum\limits_{i\in J^k_+}\frac{1}{w(f_i)}+n_{2j-1}\cdot\frac{1}{m_{2j_k}}  \big)\\
  &  \le &
 \frac{C}{m_{2j-1}^2}.
  \end{eqnarray*}
 Thus, setting $I_1=I\cap \{p-1,p,r\}$ and
 $I_2=I\setminus\{p-1,p,r\}$ we get that
 \begin{equation}\label{eq6}
   |(f\cdot \phi_2)\big(\sum\limits_{k\in
 I}(-1)^{k+1}x_k\big)|\le \sum\limits_{k\in I_1}\|x_k\|+
\sum\limits_{k\in I_2} \frac{C}{m_{2j-1}^2}\le 6C
+\frac{C}{m_{2j-1}^2}\cdot \#(I).
 \end{equation}

 From \eqref{eq5}, \eqref{eq6} we conclude that
  \[|(f\cdot \phi)\big(\sum\limits_{k\in
 I}(-1)^{k+1}x_k\big)|\le
 \frac{5C}{m_{2j-1}^2}\cdot \#(I) +3C\cdot Q_{2j-1}.\]
   \end{proof}

  \begin{proposition}\label{P3}
  If $(x_k,x_k^*)_{k=1}^{n_{2j-1}}$ is  a $(C,2j-1)$ dependent
  sequence, then
  \[
  \|\frac{1}{n_{2j-1}}\sum\limits_{k=1}^{n_{2j-1}}(-1)^{k+1}x_k\|\le
  \frac{24C}{m_{2j-1}^2}.\]
  \end{proposition}
  \begin{proof}[\bf Proof.]
  From Remark \ref{R24} we get that the block sequence $(x_k)_{k=1}^{n_{2j-1}}$
  is a $(3C,\frac{1}{m_{2j-1}^2})$ RIS,
  hence the same holds for the sequence $\big((-1)^{k+1}x_k\big)_{k=1}^{n_{2j-1}}$.
  From Lemma \ref{L15} we  have that for every $h\in K^{2j-1}\cdot D$ and every
  subinterval $I$ of the interval $\{1,2,\ldots,n_{2j-1}\}$
  it holds that
  \[
  |h\big(\sum\limits_{k\in
 I}(-1)^{k+1}x_k\big)|\le
 \frac{5C}{m_{2j-1}^2}\cdot \#(I) +3C\cdot Q_{2j-1}.
  \]
  Hence for  $h$ and $I$ as above, with $\# (I)\ge m_{2j-1}^2 Q_{2j-1}$, we have
  that
  \[
  |h\big(\sum\limits_{k\in
 I}(-1)^{k+1}x_k\big)|\le
 8C\cdot\frac{1}{m_{2j-1}^2}\cdot \#(I).\]
    Thus the basic inequality  (Proposition \ref{P1}) with the additional assumption
  is applicable to the sequence $\big((-1)^{k+1}x_k\big)_{k=1}^{n_{2j-1}}$.

  Let $f\in D$. Then there exists $g\in W$, satisfying
   \[|f(\frac{1}{n_{2j-1}}\sum\limits_{k=1}^{n_{2j-1}}(-1)^{k+1}x_k)|
     \le 8C
     \big(g(\frac{1}{n_{2j-1}}\sum\limits_{k=1}^{n_{2j-1}}e_k)+\frac{1}{m_{2j-1}^2}\big)\]
     and such that
  the functional $g$ admits a tree
  $(g_a)_{a\in\mc{A}}$  that for every $a\in\mc{A}$
  with $g_a$ of type I and $w(g_a)<m_{2j-1}^2$, the node $a$ has at most
  $m_{2j-1}^2Q_{2j-1}$ immediate successors.
  From Lemma \ref{L3}, it follows that
  \[  |g(\frac{1}{n_{2j-1}}\sum\limits_{k=1}^{n_{2j-1}}e_k)|\le\frac{2}{m_{2j-1}^2}\]
  therefore
  \[|f(\frac{1}{n_{2j-1}}\sum\limits_{k=1}^{n_{2j-1}}(-1)^{k+1}x_k)|\le\frac{24C}{m_{2j-1}^2}.\]
  Since this happens for every $f\in D$ the conclussion follows.
    \end{proof}

   \begin{theorem}\label{th2}
  The space $\eqs_D$ is Hereditarily Indecomposable.
  \end{theorem}

  \begin{proof}[\bf Proof.]
   Let $Y,Z$ be a pair of
   block subspaces of $\eqs_D$, and let $\delta>0$.
   We choose $j\in\N$ with
   $m_{2j-1}>\frac{96}{\delta}$.

   Using Lemma \ref{L18} we may inductively select a
   $(4,2j-1)$ dependent sequence
  $(x_k,x_k^*)_{k=1}^{n_{2j-1}}$ (Definition \ref{D12}) such that
  $x_{2k-1}\in Y$ and $x_{2k}\in Z$ for $1\le
  k\le\frac{n_{2j-1}}{2}$.
   We set
   \[
   y=\frac{1}{n_{2j-1}}\sum\limits_{k=1}^{n_{2j-1}/2}x_{2k-1}\in Y
  \quad\mbox{ and }\quad
   z=\frac{1}{n_{2j-1}}\sum\limits_{k=1}^{n_{2j-1}/2}x_{2k}\in
   Z.\]

 Since the  functional
 $\phi_\chi=\frac{1}{m_{2j-1}}\sum\limits_{k=1}^{n_{2j-1}}x_k^*$
 satisfies $\|\phi_\chi\|\le 1$ we get that
  $\|y+z\|\ge \phi_\chi(y+z) =\frac{1}{m_{2j-1}}$.
  On the other hand Proposition \ref{P3} implies that
  $\|y-z\|\le\frac{96}{m_{2j-1}^2}$.

  Therefore $\|y-z\|\le \delta\cdot\|y+z\|$ and this finishes the proof
  of the theorem.
  \end{proof}

  \section{The Banach algebras $\eqs_D^*$, $\mc{L}_{\diag}(\eqs_D)$ are HI}

  In this section, we initially prove that the  space
  $(\eqs_D)_*=\overline{\spann}\{e_n^*:\;n\in\N\}$ is HI.
  This, in conjunction to the fact that $\dim(\eqs_D^*/(\eqs_D)_*)=1$,
   entails that $\eqs_D^*$ is also
  HI.  From Proposition \ref{P6} we know  that the Banach algebras
   $\mc{L}_{\diag}(\eqs_D)$ and $\eqs_D^*$ are isometric,
   hence  $\mc{L}_{\diag}(\eqs_D)$ is also HI.
  We also notice (as follows from
  Remark \ref{Nrem0007}) that the algebra $\mc{K}_{\diag}(\eqs_D)$, i.e. the
  algebra
  of  compact diagonal operators on the space
  $\eqs_D$, is isometric to the space $(\eqs_D)_*$.

 \begin{definition}\label{D5} Let $C>1$ and $k\in \N$. A finitely supported
 vector $x^*\in (\eqs_D)_*$ is said to be a $C-c_0^k$ vector if $\|x^*\|\le 1$ and
 $x^*$
 takes the form $x^*=x_1^*+x_2^*+\cdots+x_k^*$ with
 $x_1^*<x_2^*<\cdots<x_k^*$ and $\|x_i^*\|\ge C^{-1}$.
 \end{definition}

 \begin{lemma}\label{L11} Let $Z$ be a block subspace of $(\eqs_D)_*$ and
 let $N\in\N$. Then there exists a block sequence
 $(x_n^*)_{n\in\N}$ in $Z$ with $\|x_n^*\|\ge 1$ such that for
 every $I\in \N^{[N]}$ and every choice of signs $(\e_i)_{i\in
 I}\in \{-1,1\}^I$ we have that $\|\sum\limits_{n\in
 I}\e_nx_n\|< 2$.\\
 (For an infinite set $L$ we denote by $L^{[N]}$
 the set of all subsets of $L$ having $N$ elements and by $[L]$
 the set of all infinite subsets of $L$.)
 \end{lemma}
 \begin{proof}[\bf Proof.]
 Assume that the lemma fails.  We select $s,j$ with $2^s>m_{2j}$
 and $N^s\le n_{2j}$.
   We shall denote by $\N_0$ the set
 $\{0,1,2,\ldots\}$.
 We choose an arbitrary normalized block sequence
 $(f^0_i)_{i\in\N_0}$ in the block subspace $Z$.

  We set
 \[ \mc{A}_1=\{L\in [\N_0],\; L=\{l_i:\; i\in \N_0\}:\;
 \forall(\e_i)_{i=0}^{N-1}\in\{-1,1\}^N, \;\;
 \|\sum\limits_{i=0}^{N-1}\e_if^{0}_{l_i}\|<2  \} \]
 \begin{eqnarray*} \mc{B}_1&=&[\N_0]\setminus \mc{A}_1\\
 &=&\{L\in [\N_0],\; L=\{l_i:\; i\in \N_0\}:\;
 \exists (\e_i)_{i=0}^{N-1}\in\{-1,1\}^N, \;\;
 \|\sum\limits_{i=0}^{N-1}\e_if^{0}_{l_i}\|\ge 2  \}.
 \end{eqnarray*}
 From Ramsey' s theorem,  there exists a homogenous set $L$ either
 in $\mc{A}_1$ or in $\mc{B}_1$. Our assumption on the failure of
 the lemma rejects the first alternative, hence the homogenous set
 is in $\mc{B}_1$. We may assume that $L=\N_0$.  In particular we
 get that there exist $\e^0_i\in\{-1,1\}$, $i\in \N_0$, such that
 setting $f^1_n=\sum\limits_{i=nN}^{(n+1)N-1}\e^0_if^0_i$, $n\in
 N_0$, we have that $\|f^1_n\|\ge 2$ for all $n$.

  We set
 \[ \mc{A}_2=\{L\in [\N_0],\; L=\{l_i:\; i\in \N_0\}:\;
 \forall(\e_i)_{i=0}^{N-1}\in\{-1,1\}^N, \;\;
 \|\sum\limits_{i=0}^{N-1}\e_if^{1}_{l_i}\|<2^2  \} \]
 \begin{eqnarray*}
  \mc{B}_2& = & [\N_0]\setminus \mc{A}_2\\
   & = &\{L\in [\N_0],\; L=\{l_i:\; i\in \N_0\}:\;
 \exists (\e_i)_{i=0}^{N-1}\in\{-1,1\}^N, \;\;
 \|\sum\limits_{i=0}^{N-1}\e_if^{1}_{l_i}\|\ge 2^2  \}.
 \end{eqnarray*}
 Again, the homogenous set $L$  resulting from Ramsey' s theorem
 can not be in $\mc{A}_2$, since then the sequence
 $(\frac{1}{2}f^1_n)_{n\in L}$ would satisfy the conclusion of the
 lemma and this contradicts to our assumption that the lemma fails.
 As before, we may assume that $L=\N_0$; we  choose
 $\e^1_i\in\{-1,1\}$, $i\in \N_0$, such that
 the functionals $f^2_n=\sum\limits_{i=nN}^{(n+1)N-1}\e^1_if^1_i$, $n\in
 N_0$ satisfy $\|f^2_n\|\ge 2^2$. Notice that
 $f^2_n=\sum\limits_{i=nN^2}^{(n+1)N^2-1}\e^0_i\e^1_{[\frac{i}{N}]}f^0_i$.

 After $s$ consecutive applications of the same argument we obtain
 a block sequence $(f^s_n)_{n\in N_0}$ with $\|f^s_n\|\ge 2^s$
 such that $f^s_n=\sum\limits_{i=nN^s}^{(n+1)N^s-1}\delta_if^0_i$
 for some sequence of signs $(\delta_i)_{i\in \N_0}$.
 Taking into account that $N^s\le n_{2j}$, Remark \ref{R3} implies
 that $\|\frac{1}{m_{2j}}\sum\limits_{i=nN^s}^{(n+1)N^s-1}\delta_i
 f^0_i\|\le 1$, i.e.   $\|f^s_n\|\le m_{2j}$. We thus get that
 $2^s\le\|f^s_n\|\le m_{2j}$ which contradicts to our choice of $s,j$. The proof
 of the lemma is complete.
 \end{proof}

 \begin{lemma}\label{L12} Let $Z$ be a block subspace of $(\eqs_D)_*$, $\e>0$ and
 $k\in \N$.  Then there exist $z^*$ a $2-c_0^k$ vector with $z^*\in Z$ and $z$ a
 $2-\ell_1^k$ average such that $\ran z^*=\ran z$, $z^*(z)>1$ and
 $\|z\|_G<\e$.
 \end{lemma}
  \begin{proof}[\bf Proof.]
 We choose $d$ with $\frac{9}{2}\cdot\frac{1}{d}<\e$ and we set
 $N=k\cdot(2d)$. Applying Lemma \ref{L11} we select a block
 sequence $(x^*_n)_{n\in\N}$ in $Z$, with $\|x_n^*\|>\frac{1}{2}$,
 such that for every subset $I$ of $\N$ with $N$ elements and
 every choice of signs $(\e_n)_{n\in I}\in \{-1,1\}^I$  we have
 that $\|\sum\limits_{n\in I}\e_nx_n^*\|\le 1$.
 For each $n$, we select $x_n\in \eqs_D$ with $\ran x_n=\ran x_n^*$,
 $\|x_n\|\le 1$ and $x_n^*(x_n)>\frac{1}{2}$.
 We notice that for every subset $I$ of $\N$ with
 $N$ elements and every choice of scalars $(\lambda_n)_{n\in I}$
 we have that $\|\sum\limits_{n\in
 I}\lambda_nx_n\|\ge\frac{1}{2}\sum\limits_{n\in I}|\lambda_n|$,
 due to the action of the functional $\sum\limits_{n\in
 I}\e_nx_n^*$, where $\e_n=\sgn(\lambda_n)$.

 We may assume, passing to a subsequence, that the sequence
 \seq{x}{n} is weakly Cauchy in $Y_G$,  hence the sequence of its
 successive
 differences, i.e. the sequence \seq{y}{n} defined as
 $y_n=x_{2n-1}-x_{2n}$, is weakly null in $Y_G$.
  We notice that the extreme
 points of the unit ball of the dual space $\mbox{B}_{Y_G^*}$ are
 contained in the set $\overline{G}^p=\{\pm \chi_{E}:\; E \mbox{ is an
 interval of }\N\}$. Thus in order to check the behavior of a
 block sequence in the weak topology in $Y_G$, it is enough to
 check the action of $\pm\chi_{\N}$.
 Passing to a further subsequence we may assume that
 $\sum\limits_{n=1}^{\infty}|\chi_{\N}(y_n)|<\frac{1}{2}$.

 We claim that $\|\sum\limits_{i=1}^m\e_iy_{k_i}\|_G\le
 \frac{9}{2}$ for every $m\in\N$, $k_1<\cdots<k_m$ in $\N$ and
 every choice of signs $\e_1,\ldots,\e_m\in\{-1,1\}$. Indeed, let
 $E$ be any finite interval. We denote by $r$ (resp. $s$) the
 minimum (resp. maximum) integer $i$ such that $E\cap\ran
 y_{k_i}\neq\emptyset$. Then for $r<i<s$ we have that
 $\chi_E(y_{k_i})=\chi_{\N}(y_{k_i})$ hence
 \begin{eqnarray*}
 |\chi_E(\sum\limits_{i=1}^m\e_iy_{k_i})|
     & \le  & \|Ey_r\|_G+  | \chi_{\N}(\sum\limits_{i=r+1}^{s-1}
     \e_iy_{k_i})\| + \|Ey_s\|_G   \\
   & \le  & \|x_{2r-1}\|_G+\|x_{2r}\|_G+
     \sum\limits_{i=r+1}^{s-1}|\chi_{\N}(y_{k_i})|
    +   \|x_{2s-1}\|_G+\|x_{2s}\|_G   \\
    & \le  & \|x_{2r-1}\|+\|x_{2r}\|+
     \sum\limits_{n=1}^{\infty}|\chi_{\N}(y_{n})|
    +   \|x_{2s-1}\|+\|x_{2s}\|\\
  &     < &1+1+\frac{1}{2}+1+1=\frac{9}{2}
 \end{eqnarray*}
 (with the obvious adaptations in the previous proof if $r=s-1$ or
 $r=s$).

 For $i=1,\ldots, k$ we set
 \[  z_i^*=\sum\limits_{l=(i-1)d+1}^{id}(x^*_{2l-1}-x^*_{2l})
 \mbox{\quad and \quad}
 z_i=\frac{1}{d}\sum\limits_{l=(i-1)d+1}^{id}y_l
 =\frac{1}{d}\sum\limits_{l=(i-1)d+1}^{id}(x_{2l-1}-x_{2l}).\]

 For each $i$ we have that $\|z_i^*\|\ge \|-x^*_{2id}\|>
 \frac{1}{2}$ (due to the bimonotonicity of the norm), $\|z_i\|\le
 \frac{1}{d}\sum\limits_{l=(i-1)d+1}^{id}(\|x_{2l-1}\|+\|x_{2l}\|)\le
 2$, while \[\|z_i\|_G\le
 \frac{1}{d}\|\sum\limits_{l=(i-1)d+1}^{id}(x_{2l-1}-x_{2l})\|_G\le
 \frac{1}{d}\cdot \frac{9}{2}<\e.\] We also have that
 \[z_i^*(z_i)=\frac{1}{d}\sum\limits_{l=(i-1)d+1}^{id}\big(x^*_{2l-1}(x_{2l-1})+x^*_{2l}(x_{2l})\big)
  >\frac{1}{d}\sum\limits_{l=(i-1)d+1}^{id}(\frac{1}{2}+\frac{1}{2})=1\]
  and $\ran z_i^*=\ran z_i$.

  Finally, we set
  \[ z^*=\sum\limits_{i=1}^{k}z_i^*
  \mbox{\quad and \quad}
   z=\frac{1}{k}\sum\limits_{i=1}^kz_i.         \]
  The fact that the functional $z^*$ is the sum of $k\cdot (2d)=N$
  functionals $\pm x_n^*$ and our initial choice of the sequence
  sequence $(x_n^*)_{n\in\N}$, imply that $\|z^*\|\le 1$ while, since
   $\|z_i^*\| \ge \frac{1}{2}$ for each $i$, we get that $z^*$ is a
   $2-c_0^k$ vector belonging to the block subspace $Z$.
  We also have that $z^*(z)=\frac{1}{k}\sum\limits_{i=1}^kz_i^*(z_i)>1$
  and $\ran z^*=\ran z$. Since $\|z\|\ge z^*(z)>1$ and $\|z_i\|\le
  2$ for $i=1,\ldots ,k$ the vector $z$ is a $2-\ell_1^k$ average,
  with $\|z\|_G\le \frac{1}{k}\sum\limits_{i=1}^k\|z_i\|_G<\e$.
  \end{proof}

  \begin{corollary}\label{C3}
  Let $Z$ be a block subspace of $(\eqs_D)_*$, $k\in\N$ and $\e,\delta>0$.
  Then there exist  $z$ a $2-\ell_1^k$ average with $\|z\|_G<\e$
  and $f\in D$ with $\dist(f,Z)<\delta$, such that
  $\ran f=\ran z$ and  $f(z)>1$.
  \end{corollary}
  \begin{proof}[\bf Proof.]  Let $z$ and $z^*$ be the $2-\ell_1^k$
  average and the $2-c_0^k$ vector respectively resulting from
  Lemma \ref{L12}. Since the norming set $D$ is pointwise dense in
  the unit ball of the dual space, we may choose $f\in D$ with
  $\ran f=\ran z^*$ such that
  $\|f-z^*\|<\min\{\delta,\frac{z^*(z)-1}{2}\}$.
  It is easy to check that $z$ and $f$ satisfy the conclusion of
  the corollary.
  \end{proof}

  \begin{lemma}\label{L19}
  Let $Z$ be a block subspace of $(\eqs_D)_*$, $j\in N$ and $\delta>0$.
  Then there exists a $(4,2j)$ exact pair $(z,z^*)$,
  with $\dist(z^*,Z)<\delta$.
  \end{lemma}
  \begin{proof}[\bf Proof.]
  Using Corollary \ref{C3}, we may choose a sequence $(z_k,z_k^*)_{k\in\N}$
  such that:
  \begin{enumerate}
  \item[(i)] The sequence \seq{z}{k} is a block sequence in $\eqs_D$
   with $\|z_k\|_G<\frac{1}{m_{2j}^3}$ and each
  $z_k$ is a $2-\ell_1^{n_{j_k}}$  for an increasing sequence
  $(j_k)_k$.
  \item[(ii)] $z_k^*\in D$ with $\dist(z_k^*,Z)<\frac{\delta}{n_{2j}}$.
  \item[(iii)] $\ran z_k^*=\ran z_k$ and $z_k^*(z_k)>1$.
  \end{enumerate}

  From Remark \ref{R5} we may assume, passing to a subsequence,
  that $(z_k)_{k=1}^{n_{2j}}$ is a $(4,\frac{1}{m_{2j}^3})$ RIS.
  We set
  \[ z^*=\frac{1}{m_{2j}}(z_1^*+z_2^*+\cdots+z_{n_{2j}}^*).  \]
  Then the functional $z^*\in D$ is of type I, with $w(z^*)=m_{2j}$ and
  $\dist(z^*,Z)\le
  \frac{1}{m_{2j}}\sum\limits_{k=1}^{n_{2j}}\dist(z_k^*,Z)<\delta$.

  From Corollary \ref{C6}, for $f\in D$ of
  type I,  we have that
 \[ |f\big(\frac{1}{n_{2j}}\sum\limits_{k=1}^{n_{2j}}z_k\big)|\le
 \left\{ \begin{array} {l@{\quad} l}
           \frac{3\cdot 4}{w(f)m_{2j}} & \mbox{ if }w(f)<m_{2j} \\[4mm]
   4(\frac{1}{w(f)}+\frac{1}{m_{2j}^2}) & \mbox{ if }w(f)\ge
   m_{2j}.
   \end{array}\right.     \]
 In particular
 $\|\frac{1}{n_{2j}}\sum\limits_{k=1}^{n_{2j}}z_k\|\le \frac{8}{m_{2j}}$.
 On the other hand
 \[z^*\big(\frac{1}{n_{2j}}\sum\limits_{k=1}^{n_{2j}}z_k\big)=\frac{1}{m_{2j}}\cdot
 \frac{1}{ n_{2j}}\cdot\sum\limits_{k=1}^{n_{2j}}z_k^*(z_k)>\frac{1}{m_{2j}}.\]  Thus there
  exists $\theta$, with $\frac{1}{8}\le \theta <1$, such that
  $z^*\big(\theta\frac{m_{2j}}{n_{2j}}\sum\limits_{k=1}^{n_{2j}}z_k\big)=1$.
  We set
  \[  z=\theta\frac{m_{2j}}{n_{2j}}(z_1+z_2+\cdots+z_{n_{2j}}) .\]

 Then $z^*(z)=1$,
 $\|z\|_G\le\frac{m_{2j}}{n_{2j}}\sum\limits_{k=1}^{n_{2j}}\|z_k\|_G<\frac{1}{m_{2j}^2}$,
 while for $f\in D$ of type I  we have that
 \[ |f(z)|\le
 \left\{ \begin{array} {l@{\quad} l}
           \frac{3\cdot 4}{w(f)} & \mbox{ if }w(f)<m_{2j} \\[4mm]
   4(\frac{m_{2j}}{w(f)}+\frac{1}{m_{2j}}) & \mbox{ if }w(f)\ge
   m_{2j}.
   \end{array}\right.     \]

  Therefore    $(z,z^*)$ is  a $(4,2j)$ exact pair
  (Definition \ref{D7})
  with $\dist(z^*,Z)<\delta$.
  \end{proof}

 \begin{theorem}\label{th5}
 The Banach space $(\eqs_D)_*$ is Hereditarily Indecomposable.
 \end{theorem}
 \begin{proof}[\bf Proof.]
 Let $Y,Z$ be a pair of block subspaces of $(\eqs_D)_*$ an let $j\in\N$.
 Using Lemma \ref{L19} we may find a
   $(4,2j-1)$ dependent sequence
  $(x_k,x_k^*)_{k=1}^{n_{2j-1}}$ (see Definition \ref{D12}) which satisfies
  $\sum\limits_k\dist(x^*_{2k-1},Y)<1$ and
  $\sum\limits_k\dist(x^*_{2k},Z)<1$.

   We set
   \[
   y^*=\frac{1}{m_{2j-1}}\sum\limits_{k=1}^{n_{2j-1}/2}x^*_{2k-1}
  \quad\mbox{ and }\quad
   z^*=\frac{1}{m_{2j-1}}\sum\limits_{k=1}^{n_{2j-1}/2}x^*_{2k}.\]

  The  functional
 $\phi_\chi=\frac{1}{m_{2j-1}}\sum\limits_{k=1}^{n_{2j-1}}x_k^*$
 satisfies $\|\phi_\chi\|\le 1$ i.e.
  $\|y^*+z^*\|\le 1$.
  Proposition \ref{P3} entails that
  $\|\frac{1}{n_{2j-1}}\sum\limits_{k=1}^{n_{2j-1}}(-1)^{k+1}x_k\|\le\frac{96}{m_{2j-1}^2}$
  while\\
  $(y^*-z^*)(\frac{1}{n_{2j-1}}\sum\limits_{k=1}^{n_{2j-1}}(-1)^{k+1}x_k)=\frac{1}{m_{2j-1}}$,
  therefore $\|y^*-z^*\|\ge \frac{m_{2j-1}}{96}$.

  Selecting $f_Y\in Y$ with $\|f_Y-y^*\|<1$ and
  $f_Z\in Z$ with $\|f_Z-z^*\|<1$, we get that $\|f_Y+f_Z\|<3$ and
  $\|f_Y-f_Z\|>\frac{m_{2j-1}}{96}-2$.
  Since this procedure may be done for arbitrary large $j$, we conclude
  that the space $(\eqs_D)_*$ is Hereditarily Indecomposable.
    \end{proof}

   \begin{theorem}\label{th3}
  The Banach algebras $\eqs_D^*$ and $\mc{L}_{\diag}(\eqs_D)$ are Hereditarily Indecomposable.
   \end{theorem}
  \begin{proof}[\bf Proof.]
 From  Proposition \ref{P5}
 the quotient space $\eqs_D^*/(\eqs_D)_*$ has dimension equal to
 one. Thus the fact that $(\eqs_D)_*$ is Hereditarily
 Indecomposable (Theorem \ref{th5}) implies that $(\eqs_D)^*$ is also Hereditarily
 Indecomposable (see also Theorem 1.4 of \cite{AT1}). As $\mc{L}_{\diag}(\eqs_D)$
 is isometric to $\eqs_D^*$ we conclude that the Banach algebra $\mc{L}_{\diag}(\eqs_D)$
 is also Hereditarily Indecomposable.
 \end{proof}

  \begin{remark}\label{R13}
  Let  $\eqs_{D,r}$ be the Banach space defined similarly to the space $\eqs_D$, with the only difference
  concerning the first inductive step of the definition of its norming set, replacing the set
  $G=\{\pm \chi_I:\; I\mbox{ is a finite interval of }\N\}$
  with the set $G_0=\{\pm e_k^*:\; k\in\N\}$. Then the space $\eqs_{D,r}$ is reflexive and HI
  while $\eqs_{D,r}^*$ is an example of a reflexive HI Banach algebra.
  The reason we  have included the set $G$ in the norming set $D$
  of the space $\eqs_D$, is  in order  to apply Theorem \ref{th1} and
  to obtain a HI Banach algebra of diagonal operators.
  \end{remark}

  \begin{theorem}\label{th6}
  Every diagonal operator $T:\eqs_D\to \eqs_D$
   is of the form $T=\lambda I+K$ with the operator
   $K$ being compact.
   \end{theorem}
   \begin{proof}[\bf Proof.]
    From Remark \ref{Nrem0007}, the isometry $\Phi:\eqs_D^*\to\mc{L}_{\diag}(\eqs_D)$
    of Theorem \ref{th1},
    carries the predual space $(\eqs_D)_*$ onto the space $\mc{K}_{\diag}(\eqs_D)$
    of compact diagonal operators of the space $\eqs_D$.
    But since $(\eqs_D)^*=(\eqs_D)_*\oplus \spann\{\chi_\N\}$ and $\Phi(\chi_\N)=I$ we
    get that
     $\mc{L}_{\diag}(\eqs_D)=\mc{K}_{\diag}(\eqs_D)\oplus\spann\{I\}$
    hence every diagonal operator $T:\eqs_D\to \eqs_D$ takes the form
    $T=\lambda I+K$ with $K$ being a compact operator.
   \end{proof}

  \begin{remark}
  Since $\eqs_D$ is the dual of the space
  $(\eqs_D)_*=\overline{\spann}\{ e_n^*:\; n\in\N\}$, observing
  that every $T\in\mc{L}_{\diag}(\eqs_D,\seq{e}{n})$, being
  $w^*-w^*$ continuous, is a dual operator, we get the following.
  The correspondence
\[
  \mc{L}_{\diag}((\eqs_D)_*) \;\ni\; R \quad \longrightarrow  \quad R^* \;\in\; \mc{L}_{\diag}(\eqs_D)
  \]
  is an onto isometry, while, restricting this correspondence to
  the subalgebras of compact diagonal operators, we obtain that
   $\mc{K}_{\diag}((\eqs_D)_*)$ is isometric to
   $\mc{K}_{\diag}(\eqs_D)$, which in turn is isometric
   to $(\eqs_D)_*$ (Remark \ref{Nrem0007}). Thus we have
   established the existence of a Banach space $Y$ with a Schauder
   basis  (namely  $Y=(\eqs_D)_*$ with the basis
   $(e_n^*)_{n\in\N}$) which is naturally isometric to the space
  $\mc{K}_{\diag}(Y)$ of its compact diagonal operators, with the
  last being of codimension 1 in $\mc{L}_{\diag}(Y)$. As we have
  noticed in the introduction,
   since the basis of $Y$ is
  shrinking, the space
  $\mc{L}_{\diag}(Y)$ is naturally identified with the second dual of $\mc{K}_{\diag}(Y)$ (\cite{Se}).
  \end{remark}

 \end{document}